\documentclass{amsart}
\usepackage{amssymb, amstext, amscd, amsmath,verbatim,xcolor}
\usepackage{enumitem}
\usepackage[all,cmtip]{xy}

\numberwithin{equation}{section}
\theoremstyle{plain}
\newtheorem{theorem}[equation]{Theorem}
\newtheorem{corollary}[equation]{Corollary}
\newtheorem{proposition}[equation]{Proposition}
\newtheorem{lemma}[equation]{Lemma}
\theoremstyle{definition}
\newtheorem{remark}[equation]{Remark}
\newtheorem{example}[equation]{Example}
\newtheorem{examples}[equation]{Examples}
\newtheorem{definition}[equation]{Definition}
\newtheorem{notation}[equation]{Notation}

\newcommand{\iip}{ideal intersection property}
\newcommand{\riip}{regular ideal intersection property}

\newcommand{\bC}{{\mathbb{C}}}
\newcommand{\bbC}{\bC}
\newcommand{\bD}{{\mathbb{D}}}

\newcommand{\bF}{{\mathbb{F}}}

\newcommand{\bN}{{\mathbb{N}}}

\newcommand{\bP}{{\mathbb{P}}}

\newcommand{\bT}{{\mathbb{T}}}

\newcommand{\bZ}{{\mathbb{Z}}}

  \newcommand{\A}{{\mathcal{A}}}

  \newcommand{\N}{{\mathcal{N}}}

\renewcommand{\phi}{\varphi}
\newcommand{\upchi}{{\raise.35ex\hbox{\ensuremath{\chi}}}}

\newcommand{\hull}{\operatorname{hull}}

\newcommand{\Prim}{\operatorname{Prim}}

\newcommand{\spn}{\operatorname{span}}

\newcommand{\Interior}[2]{{#1}^{\circ_{#2}}}
\newcommand{\interior}[1]{{#1}^{\circ}}

\newcommand{\inv}{^{-1}}

\newcommand{\cstar}{\hbox{$C^*$}}
\newcommand{\cstaralg}{\cstar-algebra}

\newcommand{\ca}{\mathrm{C}^*}
\renewcommand{\ca}{\cstar}

\renewcommand{\>}{\rangle}

\newcommand{\ol}{\overline}

\renewcommand{\subset}{\subseteq} 
\newcommand{\go}{G^{(0)}}

\newcommand{\susbeteq}{\subseteq} 

\providecommand{\innerprod}[1]{\left\langle #1\right\rangle}
\providecommand{\dperp}{{\perp\perp}}
\providecommand{\idealin}{\unlhd}
\providecommand{\dstext}[1]{\quad\text{#1}\quad}
\providecommand{\unit}[1]{#1^{(0)}}

\DeclareMathOperator{\bisection}{bi}
\newcommand{\Nbi}{N_{\bisection}}

\begin{document}


\title{Regular ideals, ideal intersections, and quotients}

\author[J.H. Brown]{Jonathan H. Brown}
\address[J.H. Brown]{
Department of Mathematics\\
University of Dayton\\
Dayton\\
OH 45469-2316\  U.S.A.} \email{jonathan.henry.brown@gmail.com}

\author[A.H. Fuller]{Adam H. Fuller}
\address[A.H. Fuller]{
Department of Mathematics\\
Ohio University\\
Athens\\
OH 45701\  U.S.A.}
\email{fullera@ohio.edu}

\author[D.R. Pitts]{David R. Pitts}\thanks{This work was supported by a grants from the Simons Foundation (DRP \#316952, SAR \#36563); and by the American Institute of Mathematics SQuaREs Program.
	The authors have no relevant financial or non-financial interests to disclose.}  \address[D.R. Pitts]{
  Department of Mathematics\\
  University of Nebraska-Lincoln\\
  Lincoln\\
  NE 68588\ 
  U.S.A.}  \email{dpitts2@unl.edu}

\author[S.A. Reznikoff]{Sarah A. Reznikoff}
\address[S.A. Reznikoff]{
Department of Mathematics\\
Virginia Tech\\
Blacksburg, VA 24061-1026\  U.S.A. }
\email{reznikoff@vt.edu}


\begin{abstract}
  Let $B \subseteq A$ be an inclusion of \cstaralg s.  We study the
  relationship between the regular ideals of $B$ and regular ideals of
  $A$.  We show that if $B \subseteq A$ is a regular \cstar-inclusion
  and there is a faithful invariant conditional expectation from $A$
  onto $B$, then there is an isomorphism between the lattice of
  regular ideals of $A$ and invariant regular ideals of $B$.  We study
  properties of inclusions preserved under quotients by regular
  ideals.  This includes showing that if $D \subseteq A$ is a Cartan
  inclusion and $J$ is a regular ideal in $A$, then $D/(J\cap D)$ is a
  Cartan subalgebra of $A/J$.   We provide a description of regular ideals in
  reduced crossed products $A \rtimes_r \Gamma$.
\end{abstract}

\maketitle

\section{Introduction}

An inclusion  $B\subseteq A$ of \cstaralg s has the {\em ideal intersection property} if every non-trivial ideal
 of 
  $A$ has non-trivial intersection with $B$. 
The \iip\ 
is useful for obtaining structural results for the inclusion.   Here are
a few of the many examples.
\begin{enumerate}
  \item When $B$ is a maximal abelian subalgebra of $A$, the ideal intersection property is a key ingredient for  establishing the Cuntz-Kreiger Uniqueness theorems for graph $C^*$-algebras (see for example \cite[Theorem~3.13]{NagRez2012}) and for
groupoids \cite[Theorem~3.1]{BNRSW2016}.
Further work on the ideal intersection property in groupoid C$^*$-algebras has been carried out in \cite{Bor2019} and \cite{KKLRU2021}.

\item  When there is an action of $\bT$ on $A$, and $B$ is the fixed
  point algebra, the ideal intersection property of $A$
is
used to establish the gauge-invariant uniqueness theorems for graph
algebras \cite[Theorem~2.3]{anHRae1997} and Cuntz-Pimsner algebras \cite[Theorem~6.2]{Kat2004}
among others.
\item If $G$ is a discrete group acting topologically freely on a
  locally compact Hausdorff space $X$, the ideal intersection property for
$C(X)\subseteq C(X)\rtimes_r G$ is used to obtain the simplicity result
of \cite[Theorem~2]{ArcSpie1994}.


\item The \iip\ has been used to characterize inclusions for which
  every pseudo-expectation is faithful \cite[Theorem~3.5]{PitZar2015}.
\item More recently, Pitts \cite[Theorem~5.2]{Pit2021} used the
ideal intersection property to characterize the existence of Cartan
envelopes.
\end{enumerate}

The \iip\ holds for certain crossed products.
Necessary and sufficient conditions for $C(X)\subseteq C(X)\rtimes_r G$  to have the ideal intersection property are given in \cite{Kaw2017}.   
Sierakowski \cite{Sie2010} shows that
$A\subseteq A\rtimes_r G$ has the \iip\ for an exact action of
discrete group $G$ on $C^*$-algebra $A$ whose induced action on
$\Prim(A)$ is essentially free.  Kennedy and Schafhauser
\cite{KenSch2019} further characterize when
$A\subseteq A\rtimes_r G$ has the \iip\ in terms of outerness and a
cohomological obstruction.

Unfortunately, the ideal intersection property does not pass, in
general, to quotients (see \cite[Example~4.8]{RaeGraphAlg} or
Example~\ref{ex Disk} below).  In the setting of amenable groupoid $C^*$-algebras, that is the inclusion $C_0(\go)\subseteq C^*(G)$, the ideal intersection property passes to quotients if and only if there is a correspondence between the lattice of ideals of $C_0(\go)$ and $C^*(G)$ \cite[Corollary~5.9]{BCFS2014} (for one direction see \cite[Corollary~4.9]{Ren1991}). Restricting attention to a class of inclusions where the ideal intersection property passes to quotients has been quite productive in the study of ideals of these algebras.  For example, Kumjian et al. \cite{KPRR1997} introduce condition~(K) so they could use  \cite[Corollary~4.9]{Ren1991} to study the ideals for this restricted class of graph $C^*$-algebras.  Indeed, for algebras of  row-finite  directed graphs $C^*(E)$ whose underlying graph $E$ satisfies condition~(K),  the inclusion $D\subseteq C^*(E)$, where $D$ is the (abelian) algebra generated by the range projections of partial isometries corresponding to finite paths in $E$, has the property that the ideal intersection property passes to quotients. 

Rather than restricting the types of inclusions considered, the paper
\cite{BFPR2021reg} takes a different tack, instead limiting the types
of ideals to {\em regular ideals}: a main result of
\cite{BFPR2021reg} shows that the \iip\ passes to
quotients of graph algebras by regular ideals.

For a commutative $C^*$-algebra $B$, regular ideals in $B$ correspond
exactly to regular open sets in $\widehat B$.  More generally, for an
arbitrary \cstaralg\ $B$, let $\overline B$ be the monotone completion
of $B$ and let $Z(\overline B)$ be the center of $\overline B$.
Hamana shows that the lattice of regular open sets in $\Prim(B)$ is
isomorphic to the projection lattice of $Z(\overline B)$, which in
turn is isomorphic to the lattice of all regular ideals of $B$, see
~\cite[Lemma~1.4 and Theorem~1.5]{Ham1982}.  
Hamana concludes in~\cite[p.\ 526, Remark~(b)]{Ham1982} that if $I(B)$ is the
injective envelope of $B$, and $A$ is a \cstaralg\ with
$B\subseteq A\subseteq I(B)$, then the map $J\mapsto J\cap B$ gives an
isomorphism of the Boolean algebra of regular ideals of $A$ onto the
Boolean algebra of regular ideals of $B$.  This conclusion is similar
to our Theorem~\ref{thm: 1-1 reg ideals} below, but we consider
inclusions of the form $B\subseteq A$ having certain regularity
properties and a faithful invariant conditional expectation of $A$
onto $B$ instead of assuming that $A$ lies between $B$ and its
injective envelope.

One of the main results of the present paper is Theorem~\ref{lem: int prop
  quotient}, which shows that if $B\subseteq A$ is an inclusion having
a faithful invariant conditional expectation of $A$ onto $B$, then the
ideal intersection property is preserved under quotients by regular
ideals.  This is a far reaching generalization of the main result in
\cite{BFPR2021reg}, as many inclusions of interest, such as reduced
crossed products by discrete groups, come equipped with a faithful
invariant conditional expectation.  Along the way we give a detailed
analysis of regular ideals, culminating in Theorem~\ref{thm: 1-1 reg
  ideals}, which shows that if $B\subseteq A$ is a regular inclusion
with the ideal intersection property and $E: A\to B$ is a faithful
invariant conditional expectation, then $J\to J\cap B$ is a one-to-one
map from the regular ideals of $A$ to the regular ideals of $B$.

Establishing that an inclusion is a Cartan inclusion has become
important from the lens of classification.  Indeed, Barlak and Li
\cite[Theorem~1.1]{BarLi2017} show that when $A$ is a separable and
nuclear $C^*$-algebra containing the Cartan MASA $B$, then $A$
satisfies the UCT. We show in Theorem~\ref{thm: Cartan quotient} that
quotients of Cartan embeddings by regular ideals remain Cartan.

One of the interesting aspects of our work is that the key to many of
our results is the \textit{regular ideal intersection property}:
$J\cap B=\{0\}\implies J=\{0\}$ for all {\emph{regular}} ideals $J$ in
$A$.   In Section~\ref{RIIP and IIP}, we show that for a very
large class of examples, the \riip\ and the \iip\  are equivalent.

After we circulated an earlier version of this paper, Exel \cite{Exel} proved a generalization of Theorem~\ref{thm: 1-1 reg ideals} using weaker hypotheses than are assumed here.
In particular, Exel does not assume the existence of an invariant faithful conditional expectation. 
Instead Exel assumes that $B \subseteq A$ satisfies an invariance axiom;
see \cite{Exel} for full details.

This paper is organized as follows.  We begin with a short section of
preliminaries where we introduce regular ideals.  In Section~\ref{RI
  RE} we analyze the relationship between regular ideals of $B$ and
$A$ where $B\subset A$ is a regular inclusion.  This culminates in
Theorem~\ref{thm: 1-1 reg ideals}; this result gives settings in which
the Boolean algebras of regular ideals in $A$ are isomorphic to the
Boolean algebra of regular and invariant ideals in $B$.  In
Section~\ref{Q I}, we prove our main theorem, Theorem~\ref{lem: int
  prop quotient}, which shows that in the presence of a faithful,
invariant conditional expectation, the ideal intersection property is
inherited by quotients of regular ideals. We use this result to show
that the Cartan property passes to quotients by regular ideals.  We then specialize to \cstaralg s of exact
groupoids in Section~\ref{exact}. In particular, Theorem~\ref{cor: U
  reg} gives an explicit description of the regular ideals of the
(reduced) $C^*$-algebra of a twisted exact groupoid $G$ when
$C_0(G^{(0)})$ has the ideal intersection property.  Section~\ref{applications}
explores some applications of our work, including applications to
graph \cstar-algebras and reduced crossed products by discrete groups.
 In
Section~\ref{RIIP and IIP}, we give classes of regular inclusions for
which the regular ideal intersection property implies the ideal
intersection property.

\subsection*{Acknowledgments} We would like to thank the anonymous referee for their helpful feedback on an earlier draft of this work.
We also wish to thank Kang Li for pointing out an error in our original proof of Proposition~\ref{prop: int = E}.


\section{Preliminaries on regular ideals}
The focus of this paper is the family of  regular ideals in a \cstar-algebra.
We will recall the definition from \cite{Ham1982} after we introduce a piece of notation.
Let $A$ be a $\ca$-algebra and let $X \subseteq A$.
Denote by $X^\perp$ the set
$$ X^\perp = \{ a\in A \colon ax = xa = 0 \text{ for all }x\in X\}. $$
If $B$ is a $\ca$-algebra such that $X \subseteq B \subseteq A$ we will write $X^{\perp_B}$ to denote the set determined by the  operation restricted to $B$:
$$ X^{\perp_B} = X^{\perp} \cap B = \{ b\in B \colon bx = xb = 0 \text{ for all }x\in X\}. $$
We can now define what it means for an ideal to be \emph{regular}.

\begin{definition}
Let $A$ be a C$^*$-algebra. 
If $J \subseteq A$ is a subalgebra satisfying $aJ \cup Ja \subseteq A$ for all $a\in A$, we will call $J$ an \emph{algebraic ideal} of $A$.
If $J \subseteq A$ is a closed algebraic ideal we will call $J$ an \emph{ideal} of $A$.

We use the notation $J \unlhd A$ to say that $J \subseteq A$ is an ideal of $A$.
\end{definition}

\begin{definition}
An ideal $J \unlhd A$ is a \emph{regular ideal} in $A$ if $J = J^{\perp\perp}$.
\end{definition}

Hamana \cite{Ham1982} gives two characterizations of the regular ideals in a $\ca$-algebra $A$: one in terms of the topology on the primitive ideal space $\Prim(A)$, and one in terms the monotone closure $\ol{A}$ of $A$. We end this section with his  characterization involving the primitive ideal space, as we require it in the sequel.

Let $X$ be a topological space.
For an open  set $U \susbeteq X$, define 
$$ U^\perp = \interior{(X\backslash U)}. $$
That is, $U^\perp$ is the interior of the complement of $U$.
Recall that an open set $U \susbeteq X$ is \emph{regular} if $U = \interior{(\ol{U})}$, i.e. $U$ is the interior of its closure.
Equivalently, an open set $U$ is regular if $U = U^{\perp \perp}$.
The regular open sets of $X$ form a complete Boolean algebra with operations:
\begin{enumerate}
    \item $U \wedge V = U \cap V$;
    \item $U \vee V = (U \cup V) ^{\perp\perp}$; and
    \item $\lnot U = U^\perp;$
\end{enumerate}
see, e.g. \cite{HalBoolBook} or \cite{GivantHalmosBook}.

\begin{remark} The ideal $C_0(U)\subseteq C_0(X)$ is regular if and only if $U$ is a regular open set in $X$.
\end{remark}

Let $A$ be a \cstar-algebra and let $\Prim(A)$ be the primitive ideal space.
For $R \subseteq A$ define
$$ \hull(R) = \{P \in \Prim(A) \colon R \subseteq P\}, $$
and for $S \subset \Prim(A)$ define
$$ \ker(S) = \bigcap_{P\in S} P. $$
Endow $\Prim(A)$ with the usual hull-kernel topology.
That is, if $S \subseteq \Prim(A)$ then the closure $\ol{S}$ is given by $\ol{S} = \hull\ker (S).$
The map $I \mapsto \Prim(A)\backslash \hull(I)$ gives a one-to-one correspondence between ideals $I \unlhd A$ and open subsets of $\Prim(A)$ \cite[Theorem~5.2.7]{MurphyBook}.
When  this map is restricted to the regular ideals of $A$, Hamana
establishes the following result.

\begin{proposition}[{c.f.~\cite[Lemma~1.4]{Ham1982}}]\label{prop: bool}
Let $A$ be a \cstar-algebra.
Then the collection of all regular ideals in $A$ forms a complete Boolean algebra with meet, join and negation given by
\begin{enumerate}
	\item $J\wedge K := J \cap K$;
	\item $J\vee K : = \<J \cup K\>^{\perp\perp}$;
	\item $\lnot J := J^\perp$.
\end{enumerate}
Furthermore, the map $I \mapsto \Prim(A)\backslash \hull(I)$ on regular ideals of a \cstar-algebra $A$ gives a complete Boolean algebra isomorphism from the regular ideals of $A$ to the regular open sets of $\Prim(A)$.
\end{proposition}


\section{Regular ideals in regular \cstar-inclusions}\label{RI RE}

\begin{definition}
Let $A$ and $B$ be $\ca$-algebras such that $B \subseteq A$.
If $B$ contains an approximate unit for $A$, we say that $B \subseteq A$ is an \emph{inclusion of $\ca$-algebras.}
We will sometimes use the notation $(A,B)$ for an inclusion of \cstar-algebras (we write the larger algebra first).
\end{definition}

\begin{remark} \label{autoau} The condition that $B$ contain an approximate unit for
  $A$ is sometimes automatic.  For example, when $B$ is maximal abelian and
  $\spn \{n\in A: nBn^*\cup n^*Bn\subseteq B\}$ is dense in $A$,
  \cite[Theorem~2.6]{Pit2021normalizers} shows that every
  approximate unit for $B$ is an approximate unit for $A$.
\end{remark}

If $B\subseteq A$ is an inclusion of $\ca$-algebras, we want to know to what extent the regular ideals of $A$ determine the regular ideals of $B$, and vice versa.
In general, one should not expect any strong relationship.
However, under some additional hypotheses, we show 
in Theorem~\ref{thm: 1-1 reg ideals} that there is an injective map from the regular ideals of $A$ to the invariant regular ideals of $B$. We will use this map in our main results on quotients by regular ideals, Theorem~\ref{lem:  int prop quotient} and Theorem~\ref{thm:  Cartan quotient}.

We  use the following notation extensively.  For $I \subseteq A$ we denote the ideal generated by $I$ in $A$ by $\<I\>_A$.
If there is no chance of confusion as to where the ideal should lie we will simply write $\<I\>$ in place of $\<I\>_A$.

The following elementary facts will be useful.

\begin{lemma}\label{lem: perp reg}
If $J \subseteq A$ is an algebraic ideal, then $J^\perp \unlhd A$ is a regular ideal.
\end{lemma}

\begin{lemma}\label{lem: reg inc}
Let $A$ be a \cstar-algebra.
Suppose $I,\ J \subseteq A$ are algebraic ideals satisfying $J^\perp \subseteq I.$
Then $J^\perp = I$ or $I \cap J \neq \{0\}.$
\end{lemma}

\begin{proof}
Suppose $J^\perp \neq I$.
Take $x \in I \backslash J^\perp$.
Since $x\notin J^\perp$, we have $xJ \neq \{0\}$ or $Jx \neq \{0\}$.
Hence $I \cap J \neq \{0\}$. 
\end{proof}

We now impose some natural extra hypotheses on the inclusions we
study.  We first show that if there is a faithful conditional
expectation from $A$ onto $B$ then $J\cap B$ is a regular ideal of $B$
whenever $J$ is a regular ideal of $A$.

\begin{lemma}\label{lem: cond perp}
Let $B \subseteq A$ be an inclusion of \cstar-algebras.
Suppose $E \colon A \rightarrow B$ is a faithful conditional expectation.
Then for any $J \unlhd A$ we have
$$J^\perp \cap B = E(J)^{\perp_B}. $$
\end{lemma}

\begin{proof}
Take $d \in E(J)^{\perp_B}$ and $a \in J$.
Then
$$ E(d^*a^*ad) = d^* E(a^*a) d =0. $$
Since $E$ is faithful, it follows that $ad=0$.
Similarly $da=0$.
Hence $d \in J^\perp$.

Suppose now that $d \in J^{\perp}\cap B$ and $a \in J$.
Then $dE(a) = E(da) = 0$, and $E(a)d=E(ad) = 0$.
Hence $d \in E(J)^{\perp_B}.$
\end{proof}

\begin{remark}
While Lemma~\ref{lem: cond perp} is stated for an ideal $J \unlhd A$, the conclusion holds for any subset $J \subseteq A$.
\end{remark}

\begin{proposition}\label{prop: int reg}
Let $B \subseteq A$ be an inclusion of \cstar-algebras.
Suppose $E \colon A \rightarrow B$ is a faithful conditional expectation.
If $J\unlhd A$ is a regular ideal of $A$, then $J\cap B$ is a regular ideal of $B$.
\end{proposition}

\begin{proof}
Let $J \unlhd A$ be a regular ideal.
Then, by regularity of $J$ and Lemma~\ref{lem: cond perp}, we get
$$ J \cap B = J^{\perp\perp} \cap B = E(J^\perp)^{\perp_B}. $$
It follows now by Lemma~\ref{lem: perp reg} that $J\cap B$ is a regular ideal in $B$.
\end{proof}

We recall the following definitions.

\begin{definition}
Let $B \subseteq A$ be an inclusion of $\ca$-algebras.
An element $n \in A$ is a \emph{normalizer of $B$} if $nBn^* \subseteq B$ and $n^*Bn \subseteq B$.
We denote the set of all normalizers of $B$ in $A$ by $\N(A,B)$.
If
$$ A = \ol{\spn}\N(A,B)$$
then we say $B$ is \emph{regular in $A$}, or $B \subseteq A$ is a \emph{regular inclusion of $\ca$-algebras}.
\end{definition}

\begin{remark}
We are now using the term {\em regular} in two different senses: one
refers to ideals and the other to inclusions. The context
will make clear which sense we intend.  This unfortunate development is
the result of historical naming conventions. In the first sense, {\em
  regular ideal}, the term goes back to at least Hamana
\cite{Ham1982}, and  seems to be derived from the notion of regular
open sets in topology.  The second, {\em regular inclusion}, goes back
to at least to  both Vershik's and Feldman and Moore's early work on
Cartan inclusions of von Neumann algebras, see~\cite{VersikNoDeOrAlOp}
and~\cite{FeldmanMooreErEqReII}.
\end{remark}

\begin{remark}
Let $B \susbeteq A$ be an inclusion of $\ca$-algebras.
Note that if $n\in \N(A,B)$, then $n^*n \in B$.
This follows from the assumption that there is an approximate unit for $B$ that also an approximate unit for $A$.
\end{remark}

\begin{definition}\label{def invariant}
Let $B \subseteq A$ be a regular inclusion of $\ca$-algebras.
Let $N \subseteq \N(A,B)$ be a $*$-semigroup with dense span in $A$.
If $K \unlhd B$ we say that
$K$ is an \emph{$N$-invariant ideal} if $nKn^* \subseteq K$ for all
$n\in N$.
When $N = \N(A,B)$ we will simply say that $K$ is an \emph{invariant ideal}.
\end{definition}

\begin{remark}\label{rem: inv ideals} If $K=I\cap B$ for some ideal $I$ in $A$ then $K$ is invariant:  indeed, for a normalizer $n$, $n In^*\subseteq I$ since $I$ is an ideal and $nBn^*\subseteq B$ since $n$ is a normalizer.

\end{remark}

\begin{definition}\label{def invariant 2}
Let $B \subseteq A$ be a regular inclusion of $\ca$-algebras.
Let $N \subseteq \N(A,B)$ be a $*$-semigroup with dense span in $A$.
If $E \colon A \rightarrow B$ is a conditional expectation, we say that $E$ is an \emph{$N$-invariant conditional expectation} if $nE(a)n^* = E(nan^*)$ for all $a \in A$ and  $n \in N$.
When $E$ is $\N(A,B)$-invariant we will simply say that $E$ is an \emph{invariant conditional expectation}.
\end{definition}

We have introduced $N$-invariance for ideals and conditional expectations because in some settings it is unclear whether a given ideal or conditional expectation is invariant under the entire $*$-semigroup of normalizers.  Nevertheless, in many situations, there is a natural $*$-semigroup $N$ of normalizers such that $\spn N$ is dense in $A$, under which an ideal or conditional expectation is invariant.
Corollary~\ref{cor: N-invariant is invariant} will show that when there is an $N$-invariant faithful conditional expectation, the $N$-invariant ideals will necessarily be invariant.
This occurs in the crossed product setting.

\begin{example}\label{ex: crossed prod}
Let $\Gamma$ be a discrete group acting via $\alpha$ on a \cstaralg\ $B$ and let $A=B\rtimes_r\Gamma$.
Let $N = \{b \delta_s \colon b \in B,\ s \in \Gamma\}$.
Then $N$ is a $*$-semigroup of normalizers with dense span in $B \rtimes_r \Gamma$.
An ideal $K$ in $B$ is $N$-invariant in the sense of Definition~\ref{def invariant} if and only if $\alpha_s(K)=K$ for all $s\in \Gamma$.

The map $E \colon \sum_{i=1}^k b_{s_i} \delta_{s_i} \mapsto b_e \delta_{e}$ extends to a faithful conditional expectation from $B\rtimes_r\Gamma$ to $B$, \cite[Proposition~4.1.9]{BrownOzawa}.
The conditional expectation $E$ is $N$-invariant.
\end{example}

Let us examine the notion of faithful, invariant
  conditional expectations for some more classes of examples.
  
   \begin{examples}\label{CEexs}
  \begin{enumerate}
    \item \label{CEexs1}
    Suppose $A$ is a unital \cstar-algebra and $B=\bC I$.   Then
  $\N(A,B)$ is the collection of all scalar multiples of unitary
  operators in $A$, so $(A,B)$ is a regular inclusion.   Here, the
  collection of all conditional
  expectations may be identified with the state space of $A$:  for a  state
  $\tau$ on
  $A$, the map $E_\tau$ given by $E_\tau(a)=\tau(a) I$ is a
  conditional expectation.    Note that $E_\tau$ is faithful
  precisely when $\tau$ is a faithful; and $E_\tau$ is invariant precisely when $\tau$ is tracial.
  Thus the existence
  of a faithful or invariant expectation is not guaranteed.  In this setting,
  there are many  conditional expectations, and there is no connection
  between  faithfulness and invariance.

\item \label{CEexs2} Suppose $B\subseteq A$ is a regular inclusion with $B$ maximal
  abelian in $A$.
\begin{enumerate}
  \item \label{CEexs2a} The existence of a conditional expectation
  is not automatic~\cite[Example~1.2]{Pit2017}.
\item \label{CEexs2b} When a conditional expectation $E: A\rightarrow B$, exists, it
  is always unique and invariant.  If $A$ is unital, uniqueness
  follows from~\cite[Theorem~3.5]{Pit2017} and invariance
  from~\cite[Proposition~3.14]{Pit2017}.  When $A$ is not unital,
  consider the unitizations, $\tilde B\subseteq \tilde A$.  By
  \cite[Corollary 2.4(i)]{Pit2021normalizers}, $\N(A,B)\subseteq \N(\tilde\A,\tilde B)$.  Thus,
  $\tilde B\subseteq \tilde A$ is a regular inclusion, and $\tilde B$ is maximal abelian in $\tilde A$.  The
  unitization $\tilde E$ of $E$ (that is,
  $A\times \bbC\ni (a,\lambda)\mapsto (E(a),\lambda)\in \tilde B$) is
  therefore an conditional expectation which is invariant and unique.
  Since $\N(A,B)\subseteq \N(\tilde\A,\tilde B)$, it follows that $E$
  is also invariant and unique.

\item \label{CEexs2c} It is possible to construct a discrete dynamical
  system $(\Gamma, X,\alpha)$ with $\Gamma$ acting freely on a
  compact Hausdorff space $X$ where the full and reduced crossed
  products of $C(X)$ by $\Gamma$ differ; in this case, $C(X)$ is
  maximal abelian in the full crossed product $C(X)\rtimes_f \Gamma$,
  yet the conditional expectation onto $C(X)$ is not faithful,
  \cite[Theorem~3.1 and Remark~3.2]{ExePitZar2021}.
\end{enumerate}
           \end{enumerate}
\end{examples}

\begin{definition}
Let $B \subseteq A$ be an inclusion of $\ca$-algebras.
We say that $B$ has the \emph{ideal intersection property in $A$} if whenever $J \unlhd A$ is a non-zero ideal, $J\cap B \neq \{0\}$.
We say that $B$ has the \emph{regular ideal intersection property in
  $A$} if whenever $J \unlhd A$ is a non-zero regular ideal, $J \cap B
\neq \{0\}$.  We will also 
say that the inclusion $(A,B)$
has the \iip\ or the \riip.
\end{definition}

Other authors use different terminology for the ideal intersection
property.  For example, an inclusion with the ideal intersection
property is called \textit{\cstar-essential} in 
\cite{PitZar2015} and is said to \textit{detect ideals} in 
\cite{KwaMey2021}.  To the best of our knowledge,
the regular ideal intersection property has not appeared elsewhere.

\begin{remark}
It is clear that if $B$ has the ideal intersection property in $A$ then it has the regular ideal intersection property.
We will show in Section~\ref{RIIP and IIP} that for  a large class of examples the regular ideal intersection property implies the ideal intersection property.
Example~\ref{ex: riip neq iip} gives an example of a regular inclusion $B \subset A$ with the regular ideal intersection property but not the ideal intersection property.
\end{remark}

\begin{notation}\label{not: J_K}
Let $B \subseteq A$ be a regular inclusion of $\ca$-algebras with a faithful conditional expectation $E \colon A \rightarrow B$.
If $K \unlhd B$ is an ideal, denote by $J_K$ the set
$$ J_K = \{ a \in A \colon E(a^*a)\in K\}. $$
\end{notation}

Our aim in the remainder of this section is to prove Theorem~\ref{thm: 1-1 reg
  ideals}, which shows that when the regular inclusion $(A,B)$ has the regular ideal intersection property and $E$ is faithful and invariant, then $K \mapsto J_K$ defines a bijection between invariant regular ideals of $B$ and regular ideals of $A$.
The following  gives a monomorphism of the invariant regular ideals of $B$
into the regular ideals of $A$.  

\begin{proposition}\label{prop: reg small to big}
Let $B \subseteq A$ be a regular inclusion of $\ca$-algebras, let  $N$ be a \mbox{$*$-semigroup} of normalizers with dense span in $A$, and let
$E \colon A \rightarrow B$ be a faithful $N$-invariant conditional
expectation.  The
following statements hold.
\begin{enumerate}
  \item  If $K \unlhd B$ is an $N$-invariant ideal, then $J_K$ is an ideal
    in $A$ such that \[J_K \cap B = K\subseteq E(J_K).\]

    \item If  $K \unlhd B$ is an invariant regular ideal, then $J_K$ is a regular ideal satisfying
$$ J_K \cap B = K = E(J_K) \dstext{and} J_K ^\perp = J_{K^{\perp_B}}.$$
\end{enumerate}
\end{proposition}

\begin{proof}
(i) Let $K \unlhd B$ be an $N$-invariant ideal.
Take $a \in J_K$ and $b \in A$.
Then $a^*b^*ba\leq \|b\|^2a^*a$ and hence 
$$ E((ba)^*ba)=E(a^*b^*ba) \leq \|b\|^2 E(a^*a) \in K $$
since conditional expectations are (completely) positive maps.  As
ideals in $\ca$-algebras are hereditary it follows that
$E((ab)^*(ab)) \in K$ and hence $ab \in J_K$.  Thus $J_K$ is a
left-ideal of $A$.  Now take any $a \in J_K$ and any $n \in N$.  Then $E((an)^*(an)) = E(n^*a^*an) = n^*E(a^*a)n \in K$, as $E$
is an $N$-invariant conditional expectation and $K$ is an $N$-invariant ideal.
Hence $an \in J_K$.
As the span of $N$ is dense in $A$, it follows that $J_K$ is
also a right ideal.  That $J_K \cap B = K \subseteq E(J_K)$ follows
from the definition, so (i) holds.

(ii) Now suppose that $K \unlhd B$ is an $N$-invariant regular ideal.
Note that $E(J_K)$ is an algebraic ideal in $B$ since $J_K$ is an ideal in $A$.
Let us show $K=E(J_K)$.

If $K \neq E(J_K)$ then, by Lemma~\ref{lem: reg inc}, $E(J_K) \cap K^{\perp_B} \neq \{0\}$.
Hence there is an $0\neq a \in J_K$ such that $E(a) \in K^{\perp_B}$.
Thus
$$ K \ni E(a^*a) \geq E(a)^*E(a) \in K^{\perp_B}. $$
Since $K$ is hereditary,  $E(a)^*E(a) \in K \cap K^{\perp_B}$ and 
hence $E(a) = 0$, contrary to choice of $a$.
Thus  $E(J_K) = K$.

It remains to show that $J_K$ is a regular ideal.
Let $J_{K^\perp} := \{a\in A: E(a^*a)\in K^{\perp_B}\}$.
It suffices to show that $J_K = J_{K^\perp}^\perp$.
Take any $a \in J_K$ and $b \in J_{K^\perp}$. Since $J_K$ is a right ideal and $J_{K^\perp}$ is a left ideal, 
 $E((ab)^*ab) \in K \cap K^{\perp_B} = \{0\}$.
Since $E$ is faithful it follows that $ab = 0$.
Similarly, $ba = 0$.
Hence $J_K \subseteq J^{\perp}_{K^\perp}$.

Now take any $a \in J_{K^\perp}^\perp$ and $b \in K^{\perp_B} \subseteq J_{K^\perp}$.
Then $E(a)b = E(ab) = 0$.
Thus $E(J_{K^\perp}^\perp) \subseteq (K^{\perp_B})^{ \perp_B}=K$; that is, $J_{K^\perp}^\perp \subseteq J_K$.
Combining the two inclusions we get $J_K = J_{K^\perp}^\perp$ and hence $J_K$ is a regular ideal.
\end{proof}

Let $B \subseteq A$ be a regular inclusion of $\ca$-algebras, let  $N$ be a \mbox{$*$-semigroup} of normalizers with dense span in $A$.
The following corollary shows that, in the presence of an $N$-invariant faithful conditional expectation $E \colon A \rightarrow B$, the $N$-invariant ideals of $B$ and the $\N(A,B)$-invariant ideals of $B$ coincide. 

\begin{corollary}\label{cor: N-invariant is invariant}
Let $B \subseteq A$ be a regular inclusion of $\ca$-algebras, let  $N$ be a \mbox{$*$-semigroup} of normalizers with dense span in $A$, and let
$E \colon A \rightarrow B$ be a faithful $N$-invariant conditional
expectation.
An ideal $K \unlhd B$ is $N$-invariant if and only if it is invariant.
\end{corollary}

\begin{proof}
Any invariant ideal of $B$ is $N$-invariant.
Suppose that $K \unlhd B$ is $N$-invariant.
By Proposition~\ref{prop: reg small to big}~(i), $K = J_K \cap B$.
Hence $K$ is an invariant ideal by Remark~\ref{rem: inv ideals}.
\end{proof}

We now return to our study of regular ideals.

\begin{corollary}\label{cor: reg small to big}
  Let $B \subseteq A$ be a regular inclusion of $\ca$-algebras,  $N$ be a $*$-semigroup of normalizers with dense span in $A$, and let
  $E \colon A \rightarrow B$ be a faithful $N$-invariant conditional
  expectation.  Further suppose that $B$ has the regular ideal
  intersection property in $A$.  If $K \unlhd B$ is an invariant
  regular ideal, then $\<K\>_A^{\perp\perp} = J_K$.
\end{corollary}

\begin{proof}
Since $\innerprod{K}_A^\dperp\idealin A$ is the smallest regular ideal
containing $K$, Proposition~\ref{prop: reg small to big}(ii) gives
$\innerprod{K}_A^\dperp \subseteq J_K$.   Let $L:=J_K\cap
\innerprod{K}_A^\perp$ and let $b\in L\cap B$.   By
Proposition~\ref{prop: reg small to big}, $J_K\cap B=K$, so $b\in K$.
But $bK=0$ because $b\in \innerprod{K}_A^\perp$.  Thus $b=0$, so
$L\cap B=\{0\}$.  
By Proposition~\ref{prop: bool}, $L$ is a regular ideal in $A$.
Since $(A,B)$ has
the \riip,  we conclude that $L=\{0\}$, that is,
$\innerprod{K}_A^\dperp =J_K$.
\end{proof}

\begin{proposition}\label{prop: int = E}
Let $B \subseteq A$ be a regular inclusion, where $B$ has the regular ideal intersection property in $A$. 
Let $E \colon A \rightarrow B$ be a faithful conditional expectation.
If $J \unlhd A$ is a regular ideal, then $J \cap B = E(J)$.
\end{proposition}

\begin{proof}
Let $I_1 = J \cap B$ and $I_2 = E(J).$
Then $I_1$ is an ideal of $B$, $I_2$ is an algebraic ideal of $B$, and $I_1 \subseteq I_2$.
The ideal $I_1 \neq \{0\}$ since $B$ has the \riip\ in $A$;
the algebraic ideal $I_2 \neq \{0\}$ since $E$ is a faithful conditional expectation.
Further, $I_1$ is an invariant ideal, and it is a regular ideal by Proposition~\ref{prop: int reg}.

Assume $I_1 \neq I_2$.
Let $I_3=I_2 \cap I_1^{\perp_B}$.
Then $I_3$ is an algebraic ideal.
Further, $I_3 \neq \{0\}$ by Lemma~\ref{lem: reg inc}.
We show 
\begin{equation}\label{I3}\<I_3\>_A \subseteq (\<I_1\>_A)^\perp.
\end{equation}
First notice that since $B$ is regular in $A$, for any algebraic ideal $I\subseteq B$, $\<I\>_A$ is the closed linear span of
$\{n_1bn_2 \colon b \in I,\ n_i \text{ normalizers}\}.$

So let $b\in I_3$, $c\in I_1$ and $\{n_i\}_{i=1}^4\subseteq \N(A,B)$.
Using the facts that $n_4^*n_4\in B$, $I_1\unlhd B$ is invariant and
$I_3\subseteq I_1^{\perp_B}$, we obtain 
\[
  ((n_1bn_2)(n_3cn_4))   ((n_1bn_2)(n_3cn_4))^*=
                                                 n_1b(n_2n_3cn_4n_4^*c^*n_3^*n_2^*)b^*n_1^*=0,\]
                                               so $
                                               (n_1bn_2)(n_3cn_4)=0$.  
  Similarly, $ (n_3cn_4) (n_1bn_2)=0$.  So \[n_1bn_2\in
  \overline\spn\{ncm: n, m\in \N(A,B), c\in
  I_1\}^{\perp}=(\<I_1\>_A)^\perp.\]  The inclusion~\eqref{I3} follows. 

Since $I_3 \subseteq I_2 = E(J)$ there exists $a \in J$ such that
$0 \neq E(a) \in I_3$.  Then
$0\neq E(a) a^*\in \<I_3\>_A\cap J$, showing
$\<I_3\>_A \cap J \neq \{0\}.$ By~\eqref{I3}, the regular ideal
$L:= J \cap \<I_1\>_A^\perp \neq \{0\}.$

Note that $\<I_1\>_A^\perp \cap B \subseteq I_1^{\perp_B}$. As $J \cap B = I_1$ it follows that $L \cap B = \{0\}$.
This contradicts the regular ideal intersection property.
Thus $J\cap B = E(J)$.
\end{proof}

Combining Proposition~\ref{prop: reg small to big} and Proposition~\ref{prop: int = E} we get the following theorem.

\begin{theorem}\label{thm: 1-1 reg ideals}
Let $B \susbeteq A$ be a regular inclusion of $\ca$-algebras satisfying the regular ideal intersection property, and let $N$ be a $*$-semigroup of normalizers with dense span in $A$.
Further assume there is an $N$-invariant faithful conditional expectation $E \colon A \rightarrow B$.
The invariant regular ideals of $B$ form a  Boolean algebra.
The map $J \mapsto J\cap B$, with inverse given by $K \mapsto J_K$, is a Boolean algebra isomorphism between the regular ideals of $A$ and the invariant regular ideals of $B$.
\end{theorem}

\begin{proof}
Recall that by Corollary~\ref{cor: N-invariant is invariant}, an ideal $K \unlhd B$ is invariant if and only if it is $N$-invariant.
Let $K_1, K_2 \unlhd B$ be invariant regular ideals.
Note that $K_1\cap K_2$ is an invariant
 ideal, so Proposition~\ref{prop: bool} shows $K_1\wedge K_2=K_1\cap
 K_2$ is an invariant regular ideal.
 For a regular invariant ideal $K \unlhd B$, Proposition~\ref{prop: reg
   small to big} shows $J_K\unlhd A$ is a regular ideal and
 $K^{\perp_B}=J_K^{\perp}\cap B$.
 But $J_K^{\perp}\cap B$ is an
 invariant ideal of $B$, whence $\neg K= K^{\perp_B}$ is also a regular invariant ideal in $B$.
As $K_1\vee K_2=\neg((\neg K_1) \wedge (\neg K_2))$, we conclude
that the invariant regular ideals of $B$ form a Boolean algebra.

The result now follows from Proposition~\ref{prop: reg small to big}
and Proposition~\ref{prop: int = E}.
\end{proof}

\begin{remark}   Combining Proposition~\ref{prop: bool}  with
  Theorem~\ref{thm: 1-1 reg ideals} shows that in the setting of a
  regular inclusion with the \riip\ and a faithful, invariant
  conditional expectation, the invariant regular ideals in $B$ and the regular open sets in $\Prim(A)$ are isomorphic Boolean algebras.
  \end{remark}


\section{Quotients by regular ideals}\label{Q I}
With the results from Section~\ref{RI RE}, we have all the tools we need to prove our main results about quotienting by regular ideals.
We first note that regular inclusions are preserved by quotients.

\begin{remark}\label{rqincl} Let $(A,B)$ be a regular inclusion and let $J\unlhd A$ be any ideal of $A$.
Then
  $(A/J, B/(J\cap B))$ is a regular inclusion.  Indeed, if
  $(u_\lambda)$ is a net in $B$ which is an approximate unit for $A$,
  then $(u_\lambda +J)$ is a net in $B/(B\cap J)$ which is an
  approximate unit for $A/J$, so $(A/J, B/(J\cap B))$ is an inclusion.
  Since $\{n+J: n\in \N(A,B)\}$ has dense span in $A/J$, the inclusion
  is regular.
\end{remark}

\begin{theorem}\label{lem: int prop quotient}
Suppose $(A,B)$ is a regular inclusion with the \iip.
Let $N$ be a $*$-semigroup of normalizers with dense span in $A$ and suppose there is  an $N$-invariant faithful conditional expectation $E \colon A \rightarrow B$.
Let $J \unlhd A$ be a regular ideal.
Then $B/(J\cap B)$ has the \iip\  in $A/J$.
\end{theorem}

\begin{proof}
For notational purposes, use $q$ for the quotient mapping of $A$ onto
$A/J$,  let $\dot A:= A/J$, and $\dot B:=B/(B\cap J)$.

Let $I \unlhd  \dot A$ satisfy $I \cap \dot B = \{0\}$ and put
 $L: = q^{-1}(I) \unlhd A$.
 Note that $L$ contains $J$.
Our task is to show $J=L$.  Arguing by contradiction, suppose $J\neq L$.  Since $J$ is a regular ideal, 
Lemma~\ref{lem: reg inc} shows $L \cap J^\perp$ is a non-zero ideal of
$A$. 
Define 
\[K:=L\cap B\] and note that $K= J\cap B$.  
Indeed,  $J\cap B \subseteq L \cap B$ because $J \subseteq L$, and 
the reverse inclusion  follows from $I \cap \dot B = \{0\}$.

By Proposition~\ref{prop: int reg},
$K$
 is a non-zero regular ideal in $B$.
By Proposition~\ref{prop: reg small to big} and Theorem~\ref{thm: 1-1 reg ideals}, $J^\perp \cap B = K^{\perp_B}$.
Hence $(L \cap J^\perp) \cap B = K \cap K^{\perp_B} = \{0\}$. 
This contradicts the ideal intersection property.
Thus $L = J$, and $I$ is the zero ideal in $A/J$.
\end{proof}

\begin{remark}\label{ss: directed graphs}
In this remark, we discuss how Theorem~\ref{lem: int prop quotient} can be used to give an alternate proof of
\cite[Theorem~3.5]{BFPR2021reg}.

Let $E=\{E^0,E^1, r, s\}$ be a directed graph.  We assume that $E$ is row-finite with no sources, that is $0<r\inv(v)< \infty$ for all $v\in E^0$.  A sequence of edges $e_1e_2\cdots e_n$ is a path in $E$, if for all $i$, $r(e_i)=s(e_{i-1})$; we say it is a return path if $r(e_1)=s(e_n)$.  We say $E$ satisfies condition~(L) if for every return path $e_1e_2\cdots e_n$ there exists an $i$ such that $r\inv(r(e_i))\backslash \{e_i\}\neq \emptyset$.  

A set of mutually orthogonal projections $\{P_v\}_{v\in E^0}$ and a set of partial isometries $\{S_e\}_{e\in E}$ is a Cuntz-Krieger $E$-family if 

\[
P_{s(e)}=S_{e}^{*} S_{e}^{}=\sum_{r(f)=s(e)} S^{}_{f}S_{f}^*.
\]
The $C^*$-algebra of $E$, $C^*(E)$, is the unique $C^*$-algebra
generated by a universal Cuntz-Krieger $E$-family, $\{p_v, s_e\}$ (for
details see \cite{RaeGraphAlg}).  The \cstar-algebra $C^*(E)$ comes equipped with a gauge action.
More precisely, for each $t \in \bT$ the map
$$ \gamma_t(s_e) = ts_e \text{ and } \gamma_t(p_v) = p_v $$
for  $e\in E^1$ and $v\in E^0$, uniquely defines an automorphism on C$^*(E)$ so that the map $t \mapsto \gamma_t$ is strongly continuous.

Let $D_E$ be the $C^*$-subalgebra of
$C^*(E)$ generated by
$$\{(s^{}_{e_1}s^{}_{e_2}\cdots s^{}_{e_n})(s^{}_{e_1}s^{}_{e_2}\cdots s^{}_{e_n} )^*: n\in \bN, e_i\in E^{1}\}.$$
A consequence of the Cuntz-Kreiger Uniqueness theorem \cite[Theorem
3.7]{KPR1998} is that $D_E\subseteq C^*(E)$ has the ideal intersection
property if and only if $E$ satisfies condition~(L).

Now if $J$ is an ideal in $C^*(E)$, then by \cite[Theorem~4.1]{BPRS2000}, $J$ is generated by a set of vertex
projections $\{p_v: v\in H\subseteq E^0\}$ if and only if $J$ is gauge-invariant and in this case
there exists a directed graph $E/J$ such
that $C^*(E/J)\cong C^*(E)/J$.

Let $E$ be a a directed graph satisfying condition~(L).  Suppose $J$
is  a  gauge-invariant regular ideal.
By Theorem~\ref{lem:  int prop quotient}, 
$$D_E/(J\cap D_E)\subset C^*(E)/J\cong C^*(E/J)$$
has the ideal intersection property and thus $E/J$ satisfies condition~(L).
 Thus \cite[Theorem~3.5]{BFPR2021reg} follows from Theorem~\ref{lem: int prop quotient}.

In  \cite[Proposition~3.7]{BFPR2021reg} we go on to show that all regular ideals in $C^*(E)$ are gauge-invariant when the graph $E$ has condition~(L).
The analogous result for higher-rank graph \cstar-algebras is proved in \cite[Proposition~6.7]{Sch2021}.
We generalize these results in Theorem~\ref{cor: U reg} below (see Section~\ref{ss: directed graphs 2} for details on the relationship to graph algebras).
\end{remark}

Theorem~\ref{lem: int prop quotient} has implications for quotients of
Cartan inclusions by regular ideals, and to this we now turn our attention.
Let us first recall the notion of a Cartan inclusion from \cite{Ren2008}.

\begin{definition}
Let $D \subseteq A$ be a regular inclusion of $\ca$-algebras.
We say that $D$ is a \emph{Cartan subalgebra}, or \emph{$D$ is Cartan in $A$} if
\begin{enumerate}
\item there is a faithful conditional expectation $E \colon A \rightarrow D$;
\item $D$ is maximal abelian in $A$.
\end{enumerate}  We will also say that $(A,D)$ is a \textit{Cartan
  inclusion} when these conditions hold.
\end{definition}

\begin{remark}\label{rem: cartan iip}
Recall that if $D$ is a Cartan subalgebra of $A$, then the inclusion $D\subseteq A$ has the ideal intersection property.
This well-known fact can be found in several places, e.g. \cite[Corollary~3.2]{NagRez2014} and \cite[Theorem~6.1]{Pit2017}.
\end{remark}
Thus for a Cartan inclusion $(A,D)$, Theorem~\ref{lem: int prop
  quotient} implies that if $J\idealin A$ is a regular ideal, then
$D/(J\cap D)\subseteq A/J$ has the ideal intersection property too.
It is therefore natural to ask if this quotient inclusion is also
Cartan.  We prove this in Theorem~\ref{thm: Cartan quotient} below.
The following example shows that this does not hold for arbitrary
ideals.  Another example for graph algebras is described in
\cite[Remark~4.5]{BPRS2000}.

\begin{example}\label{ex Disk} Suppose $\bZ$ acts on the closed disk $\ol{\bD}$ by irrational rotation and $$\bP:=\ol{\bD}\backslash \{(0,0)\}.$$ Then $n\cdot x\neq x$ for all $x\in\bP$.
Notice that $\bP$ is dense in $\ol{\bD}$ so $C(\ol{\bD})$ is Cartan in $C(\ol{\bD})\rtimes \bZ$ ( see \cite[Example~6.1]{Ren2008}). 

On the other hand,  $\bP$ is an open invariant subset of $\ol{\bD}$ and so by \cite[Proposition~II.4.6]{RenaultBook} we have that $C_0(\bP)\rtimes \bZ$ is an ideal in $C(\ol{\bD})\rtimes \bZ$.  Moreover the quotient $C(\ol{\bD})\rtimes \bZ/(C_0(\bP)\rtimes \bZ)$ is isomorphic to $C^*(\bZ)\cong C(\bT)$ and  $C(\ol{\bD})/C_0(\bP)\cong \bC$.
As $\bC$ is not maximal abelian in $C(\bT)$ we have
$C(\ol{\bD})/C_0(\bP)$ is not Cartan in $C(\ol{\bD})\rtimes
\bZ/(C_0(\bP)\rtimes \bZ)$.   Thus Cartan inclusions are not
necessarily preserved by quotients.   Notice that
$\interior{(\overline{\bP})}=\ol{\bD}$ and so $\bP$ is not regular in $\ol{\bD}$. 
Hence, by Proposition~\ref{prop: int reg}, we have $C_0(\bP)\rtimes \bZ$ is not a regular ideal in $C(\ol{\bD})\rtimes \bZ$.
\end{example}

Before proving Theorem~\ref{thm: Cartan quotient} which shows the quotient of a  Cartan
inclusion by a {\em
  regular} ideal is again a Cartan inclusion,
we show  that Theorem~\ref{lem: int prop quotient} guarantees the existence of a faithful conditional expectation in the quotient.

\begin{lemma}\label{lem: CE quotient}
Let $(A,B)$ be a regular inclusion with the \iip, let  $N$ be a $*$-semigroup of normalizers with dense span in $A$, and suppose $E:
A\rightarrow B$ is a faithful $N$-invariant conditional expectation.
If $J\idealin A$ is a regular ideal, then the map
\begin{align*}
	E_{A/J} \colon A/J &\rightarrow B/(B\cap J)\dstext{given by}\\
			a + J &\mapsto E(a) + (J \cap B)
\end{align*}
is a well-defined faithful conditional expectation.
\end{lemma}

\begin{proof}
Let $K := J \cap B$.
For $a_1, a_2\in A$,  if $a_1 + J = a_2 + J$,
then $a_1 - a_2 \in J$, whence $E(a_1 - a_2) \in E(J)$.
By Proposition~\ref{prop: int = E}, $E(J) = K$.
Thus  $E(a_1) + K = E(a_2) + K$.  Therefore, $E_{A/J}$ is  well-defined,
and an argument using~\cite[Theorem~1]{Tom1957} (adjoining a unit if necessary) shows it is a conditional expectation.

It remains to show that $E_{A/J}$ is faithful.  Let
\[L := \{ x\in A/J \colon E_{A/J}(x^*x) = 0\}.\]  We shall show that $L$
is an ideal in $A/J$.  That $L$ is a left ideal follows as in the
proof of Proposition~\ref{prop: reg small to big}.  To show $L$ is
a right ideal, we use the $N$-invariance of $E$:  for 
$a + J \in L$ and  $n \in N$, 
\begin{align*}
	E_{A/J}([(a + J)(n + J)]^*(a + J)(n + J)) & = E(n^*a^*an) + K \\
	& = nE(a^*a)n + K\\
	& = 0.
\end{align*}
This gives $(a + J)(n+J) \in L$.  Therefore for any $y\in
\spn N$, $(a+J)(y+J)\in L$. Since $\spn N$ is dense in $A$ by assumption,  $L$ is a right ideal.

By definition, $L \cap (B/K) = \{0\}$, and an application of
Theorem~\ref{lem: int prop quotient} shows $L=\{0\}$.  Thus
$E_{A/J}$ is a faithful conditional expectation.
\end{proof}

Here is the promised result concerning quotients of Cartan inclusions
by regular ideals.

\begin{theorem}\label{thm: Cartan quotient}
If $D$ is a Cartan subalgebra of a $\ca$-algebra $A$ and $J \unlhd A$ is a regular ideal, then $D/(J\cap D)$ is a Cartan subalgebra of $A/J$.
\end{theorem}

\begin{proof}
Let $K = D \cap J$.
Remark~\ref{rqincl} shows $(A/J, D/(D\cap J))$ is a regular inclusion.

Since $(A,D)$ is a Cartan inclusion, it has the ideal intersection
property (see Remark~\ref{rem: cartan iip}).  Further, by Example~\ref{CEexs}(ii)(b) the faithful
conditional expectation $E \colon A \rightarrow D$ is invariant.
We can thus apply Lemma~\ref{lem: CE quotient} to see that there is a
faithful conditional expectation $E_{A/J} \colon A/J \rightarrow D/K$.

It remains to show that $D/K$ is a maximal abelian subalgebra of $A/J$.
To this end, take $a \in A$ such that
$ad - da \in J$ for all $d \in D$.  To complete the proof, we will
show
\begin{equation}\label{Car1}
  a + J = E(a) + J.
\end{equation}

Note that $J^\perp \cap D = K^{\perp_D}$ by Proposition~\ref{prop: reg small to big}(ii). 
If $e \in J^\perp \cap D$
then $ae - ea \in J^\perp \cap J$.
Thus $ae - ea = 0$.
This shows that  \begin{equation*}\label{Car2}
  ae = ea\text{ for all } e \in K^{\perp_D}.
\end{equation*}

Next we show $ea$ commutes with $D$.  For $d \in D$ and $e \in
K^{\perp_D}$, the facts that
$ad-da \in J$, $D$ is abelian, and $ae=ea$ yield,
\[
	0 = e(ad-da) 
	= (ea)d - d(ea).
\]
Hence $ea$ commutes with $D$.

As $D$ is maximal abelian in $A$ we have $ea=ae \in D$, whence
$ae = E(ae) = E(a)e$.  These considerations are independent of the
choice of $e\in K^{\perp_D}$, so \begin{equation*}\label{Car3} (a - E(a))e
  = 0\text{ for all } e \in K^{\perp_D}.
\end{equation*}

Let $n_1$ and $n_2$ be normalizers of $D$ and take $e \in K^{\perp_D}$.
Then
\begin{equation*}
((a - E(a))n_1e n_2)((a - E(a))n_1e n_2)^*
= (a - E(a))n_1 e n_2 n_2^* e^* n_1^*(a - E(a))^* = 0,
\end{equation*}
since $n_1 e n_2 n_2^* e^* n_1^* \in K^{\perp_B}$.
Thus $(a - E(a))n_1 e n_2 = 0$.
Therefore \[a - E(a) \in \<K^{\perp_B}\>^\perp.\]

 However, $\<K^{\perp_B}\>^\perp = J$ by Proposition~\ref{prop: reg small to big}.
Hence $a + J = E(a) + J$, establishing~\eqref{Car1}.
\end{proof}

\begin{remark}
The converse of Theorem~\ref{thm: Cartan quotient} does not hold: it is possible to find a Cartan inclusion $(A,D)$  and a non-regular ideal $J \unlhd A$ so that $D/(J \cap D)$ is Cartan in $A/J$, see e.g. \cite[Example~3.9]{BFPR2021reg}.
           \end{remark}

\section{\cstaralg s of 
  Twisted \'Etale Groupoids}\label{exact}

We now specialize our study to the $C^*$-algebras of twisted,
Hausdorff, \'etale groupoids.  As a consequence of exactness we will
get an explicit description of the regular ideals of such
$C^*$-algebras.

To begin our discussion, we recall a few facts and some notation
concerning $C^*$-algebras from twists.  We refer the reader to
\cite[Section~2]{BFPR2021graded} for details.

For a groupoid $G$ we will use $\go$ to denote its unit space.  We
identify $\go$ with the objects in $G$ and use $r,s: G\to
\go$ to denote the range and source maps.

For each subset $\Delta\subset \go$ there is a subgroupoid 
\[
G_\Delta:=\{\gamma\in G: r(\gamma),s(\gamma)\in \Delta\}.
\]
When $G_\Delta=\{\gamma\in G: s(\gamma)\in \Delta\}$, that is,
$s(\gamma)\in \Delta\Rightarrow r(\gamma)\in \Delta$, $\Delta$ is said
to be an \textit{invariant subset} of $\go$; we will also  refer to
such as set as \emph{invariant}.

We assume groupoids are endowed with a locally compact Hausdorff
topology in which composition and inversion are continuous.  
An open set $B \subseteq G$ is a \emph{bisection} if $r|_B$ and $s|_B$ are homeomorphisms onto their images. 
We say $G$ is \emph{\'etale} if it has a basis consisting of bisections.
A \emph{twist} is a central extension of a groupoid $G$
\[
  \go\times \bT\overset{\iota}{\to} \Sigma\overset{q}{\twoheadrightarrow} G.
\]
We usually suppress writing the maps $\iota$ and $q$ and simply
denote a twist by $\Sigma \rightarrow G$ or $(\Sigma;G)$.
The map $q$ restricts to a homeomorphism between the unit spaces $\Sigma^{(0)}$ and $G^{(0)}$.
We will thus identify $\Sigma^{(0)}$ with $\go$ without mention of $q$.
A well-known construction produces the \textit{reduced $C^*$-algebra,
  $C^*_r(\Sigma;G)$ of $(\Sigma;G)$};  it 
is the completion of a convolution algebra using the reduced norm.   
Details may be found in a number of sources; see for example,
\cite{Ren2008}, \cite{SimsGroupoidsBook} or  
\cite{BFPR2021graded}.   
If $\Sigma = G\times \bT$ we drop $\Sigma$ from the notation and refer to this algebra as $C^*_r(G)$.  

If $H$ is an open subgroupoid of $G$, the twist $\Sigma\rightarrow G$
determines a central extension 
\[
  H^{(0)}\times \bT\to q^{-1}(H)\twoheadrightarrow H
\]
and $C_r^*(q^{-1}(H);H)$ embeds in $C^*_r(\Sigma;G)$
\cite[Lemma~2.7]{BFPR2021graded}. We will use this embedding in the
following instances.
\begin{enumerate}
\item If $U$ is an invariant open subset of $\go$, then $G_U$ is open
  in $G$ and $q^{-1}(G_U) = \Sigma_U$.  The C$^*$-algebra
  $C^*_r(\Sigma_U;G_U)$ embeds in $C_r^*(\Sigma; G)$ as a closed
  ideal.  In fact,
  $C_r^*(\Sigma_U; G_U) = \<C_0(U)\>_{C_r^*(\Sigma;G)}$ by
  \cite[Lemma~2.7]{BFPR2021graded}.
\item Since $G$ is \'etale, the unit space $\go$ is clopen in $G$ and
  $C^*_r(q^{-1}(\go); \go)$ embeds in $C^*_r(\Sigma; G)$.  As 
  $C^*_r(q^{-1}(\go); \go)\cong C_0(\go)$, we identify $C_0(\go)$ with
  $C^*_r(q^{-1}(\go); \go)$.  In this sense, we regard $C_0(\go)$ as a
  subalgebra of $C^*_r(\Sigma; G)$; furthermore,
  $(C^*_r(\Sigma;G), C_0(\go))$ is a regular inclusion.
\end{enumerate}
In the last instance, there exists a canonical conditional expectation
$E: C^*_r(\Sigma; G)\to C_0(\go)$.  This conditional expectation is
invariant for the $*$-semigroup $\Nbi$ of normalizers that are supported on bisections. 
By Corollary~\ref{cor: N-invariant is invariant} an ideal $C_0(U)$ in $C_0(\go)$ is invariant in the sense of Definition~\ref{def invariant} if and only if it is $N_{\mathrm{bi}}$-invariant.
Moreover, $C_0(U)$ is an invariant ideal of $C_0(\go)$  if and only if $U$ is an invariant open set.

Suppose $U$ is an open invariant set in $\go$ and $F=\go\setminus U$.
Since $F$ is a closed invariant subset, $G_F$ is a groupoid with unit
space $F$.
Further, $q^{-1}(G_F) = \Sigma_F$.
We thus get a twist,
\[F\times \bT\rightarrow \Sigma_F\twoheadrightarrow G_F.\]  It is not difficult to show
that the restriction map,
$C_c(\Sigma;G)\ni g\mapsto g|_F\in C_c(\Sigma_F;G_F)$, extends to a
$*$-homomorphism, \begin{equation}\label{phidef}
\phi: C_r^*(\Sigma;G)\to C^*(\Sigma_F; G_F).
\end{equation}
Thus, with $\theta: C_0(\go)\rightarrow C_0(F)$ denoting the quotient
map $\theta(f) = f|_F$,  we obtain the
commutative diagram (where the vertical maps are the conditional expectations described above), 
\begin{equation}
\label{comm diagram}
\xymatrix{
C_r^*(\Sigma_U;G_U)\ar[r]\ar[d]^{E_U} & C_r^*(\Sigma;G)\ar[d]^{E}\ar@{->>}[r]^\phi &C_r^*(\Sigma_F;G_F)\ar[d]^{E_F}\\
 C_0(U)\ar[r] & C_0(\go)\ar@{->>}[r]^\theta& C_0(F).
}
\end{equation}
While the bottom row is exact in the middle, it is possible that the
top row is not exact \cite[Appendix by Skandalis]{Ren1991}.

\begin{definition}\label{def: exact}
If the top row of \eqref{comm diagram} is also
exact in the middle, we
say that $G$ is {\em inner exact at $U$.}  When $G$ is inner exact for every open
invariant set $U\subseteq \go$, we say $G$ is \textit{inner exact}, or
more simply, \textit{exact}.
\end{definition}

By  \cite[Theorem~3.5]{Lal2019}, inner exactness is a property of $G$
and is independent of the twist over $G$. 
Our next goal is to obtain a description of certain regular ideals in the reduced
\cstaralg\ of a twist over an exact groupoid, see Theorem~\ref{cor: U
  reg} below.

Let $U \subseteq \go$ be an invariant open set.
Set 
$$ J_U := \{ x \in C_r^*(\Sigma; G) \colon E(x^*x) \in C_0(U)\}.$$
(This is simplified notation for the ideal $J_{C_0(U)}$ defined in Notation~\ref{not: J_K}.)

\begin{proposition}\label{Prop: J}  Suppose $\Sigma\to G$ is a twist,
$U\subseteq \go$ is an invariant open set, and $\phi$ is defined as in \eqref{phidef}. 
Then $\ker \phi=J_U$.
\end{proposition}

\begin{proof}
Suppose $x\in \ker \phi$. Then $\theta(E(x^*x))=E_F(\phi(x^*x))=0$ so
$E(x^*x)\in \ker\theta$.  Since the bottom row of~\eqref{comm diagram} is
exact, $E(x^*x)\in C_0(U)$. Thus,  $\ker\phi\subseteq J_U$.  

Now suppose $x\in J_U$.  Then $E_F(\phi(x^*x))=\theta(E(x^*x))=0$, so
faithfulness of $E_F$ gives $\phi(x)=0$.  Hence $x\in \ker\phi$, and so $\ker\phi=J_U$ as desired.
\end{proof}

This new description of $J_U$ as the kernel of a quotient map allows us to strengthen Corollary~\ref{cor: reg small to big} for twisted groupoid C$^*$-algebras.
In the groupoid setting, Corollary~\ref{cor: reg small to big} says that if $C_0(\go)$ has the regular ideal intersection property in $C_r^*(\Sigma; G)$, and $U \subseteq \go$ is an open regular invariant set, then  $C_r^*(\Sigma_U; G_U)^{\perp\perp}=J_U$.
We now show that for reduced C$^*$-algebras of twists we can drop the regular ideal intersection property.

\begin{proposition}\label{prop: reg kern}
Let $\Sigma\to G $ be a twist.  Let $U\subseteq \go$ be a regular invariant open set and $F=\go\setminus U$.  Let $\phi: C_r^*(\Sigma;G)\to C_r^*(\Sigma_F; G_F)$ be the $*$-homomorphism that extends restriction.  Then $\ker \phi$ is a regular ideal and $C_r^*(\Sigma_U; G_U)^{\perp\perp}=\ker \phi$.
\end{proposition}

\begin{proof}
  By Proposition~\ref{Prop: J}, $\ker \phi =J_{U}$.  Since $U$ is
  regular and invariant, Proposition~\ref{prop: reg small to big}
  shows that $\ker\phi$ is regular.  Since
  $ C_r^*(\Sigma_U; G_U)\subseteq \ker \phi$, we obtain
  $C_r^*(\Sigma_U; G_U)^{\perp\perp}\subseteq \ker \phi$.

  Aiming at a contradiction, assume
  $C_r^*(\Sigma_U; G_U)^{\perp\perp}\neq \ker \phi$.  By
  Lemma~\ref{lem: reg inc},
  \[\ker \phi \cap C_r^*(\Sigma_U; G_U)^{\perp}\neq \{0\}.\]   Fix a non-zero
  positive element $a\in \ker \phi \cap C_r^*(\Sigma_U; G_U)^{\perp}$.
For $e\in C_0(U)$,
$ea=0$ because $a\in C_r^*(\Sigma_U; G_U)^{\perp}$.    On the other
hand, if $f\in C_0(U^\perp)$,  then $f E(a)\in C_0(U^\perp)  \subseteq C_0(F)$, so
  $$f E(a)=\theta(f E(a))=\theta(E(f
  a))=E_F(\phi(f)\phi(a))=0.$$
Thus for any $h\in C_0(U\cup U^\perp)$
\[hE(a)=0.\]   But $U\cup U^\perp$ is dense in $\go$, so $E(a)=0$.
Faithfulness of $E$ gives $a=0$, contrary to hypothesis.
Thus
  $\ker \phi=C_r^*(\Sigma_U; G_U)^{\perp\perp}$, as desired.
\end{proof}

Let $\Sigma \to G$ be a twist and assume that $G$ is exact.
We now show that if $(C_r^*(\Sigma; G), C_0(\go))$ has the \iip\ then the regular ideals of $C^*_r(\Sigma; G)$ have an explicit description in terms of the dynamics of $G$.

\begin{theorem} \label{cor: U reg} 
Let $\Sigma \rightarrow G$ be a twist.
Suppose that $G$ is exact and $U\subseteq \go$ is a regular  invariant open set.
Then $C^*_r(\Sigma_U; G_U)$ is a regular ideal in $C^*_r(\Sigma; G)$.

If, in addition, $C_0(\go)$ has the regular ideal intersection property in $C_r^*(\Sigma; G)$, every regular ideal of $C^*_r(\Sigma; G)$ is of this form.
\end{theorem}

\begin{proof} Since $G$ is exact, $\ker \phi=C^*_r(\Sigma_U;
G_U)=J_U$. By Proposition~\ref{prop: reg kern},  $C^*_r(\Sigma_U;
G_U)$ is a regular ideal in $C^*_r(\Sigma;G)$.

Theorem~\ref{thm: 1-1 reg ideals} shows that in the presence of the
\riip, all regular ideals are of
the form $C^*_r(\Sigma_U; G_U)$.
\end{proof}

The following corollary to (non-twisted) groupoid C$^*$-algebras is a special case of Theorem~\ref{cor: U reg}.
We record it, however, as many of the examples which motivated this study are in this setting.

\begin{corollary}\label{cor: non-twisted}
Let $G$ be an exact locally compact  Hausdorff \'{e}tale groupoid.
If $U \subseteq \go$ is a regular open invariant set, then $C^*_r(G_U)$ is a regular ideal in $C_r^*(G)$.

If $C_0(\go)$ has the \iip\ inside $C_r^*(G)$ then all regular ideals of $C_r^*(G)$ are of this form.
\end{corollary}

In Section~\ref{ss: directed graphs 2}, we discuss how Corollary~\ref{cor: non-twisted} applies to graph C*-algebras.
More generally, Proposition~\ref{exactCartan} gives an application of Theorem~\ref{cor: U reg} for Cartan inclusions in nuclear C$^*$-algebras. 
In Theorem~\ref{cor: crossed prod ideals}, a result analogous to Theorem~\ref{cor: U reg} and Corollary~\ref{cor: non-twisted} for reduced crossed products of discrete group actions on (not necessarily abelian) C$^*$-algebras is given.
If a discrete group $\Gamma$ acts on a locally compact Hausdorff space $X$, then the reduced crossed product $C_0(X) \rtimes_r \Gamma$ is both an example of a reduced groupoid C$^*$-algebra and a reduced crossed-product C$^*$-algebra.
Corollary~\ref{cor: trans groups} gives an application of both Corollary~\ref{cor: non-twisted} and Theorem~\ref{cor: crossed prod ideals} in this setting.

\begin{remark}\label{rem: reg and exact} A careful reader will note
  that we did not need that $G$ was exact for Theorem~\ref{cor: U reg} or Corollary~\ref{cor: non-twisted}, only
  that Diagram~\eqref{comm diagram} is exact for regular open sets $U\subseteq
  \unit G$.
\end{remark}

Our aim in the remainder of this section is to present
Example~\ref{ex: HLS}.
This is an example of a groupoid $G$ that is not inner exact but which is inner exact at all open regular invariant sets $U \subseteq \go$.

\begin{definition}
Let $\Gamma$ be a discrete group.
A sequence of subgroups $(K_n)$ is an \emph{approximating sequence} for $\Gamma$ if
\begin{enumerate}
    \item each $K_n$ is a normal, finite index subgroup of $\Gamma$;
    \item $K_n \supseteq K_{n+1}$ for all $n$; and
    \item $\bigcap_n K_n = \{e\}$.
\end{enumerate}
\end{definition}

Given a discrete group $\Gamma$ with an approximating sequence $(K_n)$ one can construct what is known as a HLS groupoid.
These groupoids were introduced (and named for) Higson, Lafforgue and Skandalis \cite{HLS}.
Higson, Lafforgue and Skandalis use HLS groupoids to give a counter-example to the groupoid version of the Baum-Connes conjecture.
In so doing, they give an example of a HLS groupoid which is not exact.
We give now the definition of a HLS groupoid, as presented by Willett \cite{Wil2015}. 

\begin{definition}
Let $\Gamma$ be a discrete group 
 with an approximating sequence $(K_n)$.
For each $n$ let $\Gamma_n = \Gamma/ K_n$, and let $\pi_n \colon \Gamma \rightarrow \Gamma_n$ be the quotient map.
Let $\Gamma_\infty = \Gamma$ and $\pi_\infty$ be the identity map on $\Gamma$.
Let $\bN^* = \bN \cup \{\infty\}$ be the one-point compatification of $\bN$ and  let $G$ be the disjoint union of $\{n\} \times \Gamma_n$ for all $n\in \bN^*$.
That is,
$$ G = \bigsqcup_{n \in \bN^*} \{n\} \times \Gamma_n. $$
Equip $G$ with the topology generated by the sets
\begin{enumerate}
    \item $\{(n,g)\}$ for $n \in \bN$ and $g\in \Gamma_n$; and
    \item for each $\gamma \in \Gamma$ and $N \in \bN$,
      $\{(n,\pi_n(\gamma)) \colon n \in \bN \cup \{\infty\},\ n \geq
      N\}$.
\end{enumerate}
With this topology, $\{(n, \pi_n(e)):n\in\bN^*\}$ is a clopen set in
$G$, and with its relative topology, it is homeomorphic to
$\bN^*$. We identify $\{(n,\pi_n(e)): n\in \bN^*\}$ with $\bN^*$.

Endow $G$ with a partial product
\[ (n,g_1)(m,g_2) = \begin{cases} (n,g_1g_2) & \text{if }n = m\\
\text{undefined} &  \text{when }n\neq m.
\end{cases}
\]
The inverse operation is given by $(n,g)^{-1} = (n,g^{-1})$.
With these operations, $G$ becomes a Hausdorff \'etale topological groupoid and is
called the \emph{HLS groupoid} associated to $\Gamma$ and $(K_n)$.
The unit space of $G$ is $\bN^*$. 
\end{definition}

We will need the following proposition for Example~\ref{ex: HLS}.

\begin{proposition}\label{prop: HLS reg}
Let $\Gamma$ be a discrete group with an approximating sequence
$(K_n)$ and 
let $G$ be the associated HLS groupoid.
Assume that $C_r^*(G) = C^*(G)$.
Then 
$$ C_r^*(G_U) \rightarrow C_r^*(G) \rightarrow C_r^*(G_{\bN^* \backslash U})$$
is exact for all regular open  sets $U \subseteq \bN^*.$
\end{proposition}

\begin{proof}
This result hinges on two claims.

\noindent \textit{Claim 1.} \textit{If $U \subseteq \bN$, then $C_r^*(G_U) = C^*(G_U)$.}

To prove this claim note that if $U$ is finite then the groupoid $G_U$ is finite, and the result is clear.
If $U$ is infinite, then the result follows from \cite[Lemma~2.6]{Wil2015}.
Thus Claim 1 is proved.

\noindent  \textit{Claim 2.} \textit{If $U \subseteq \bN^*$ with $\infty \in U$ and $U \cap \bN$ infinite, then $C_r^*(G_U) = C^*(G_U)$.}

To prove this claim, note that $G_{U}$ is the HLS groupoid associated to $\Gamma$ and the approximating sequence $(K_n)_{n\in U\cap \bN}$.
By assumption $C_r^*(G) = C^*(G)$, so it follows from \cite[Lemma~2.7]{Wil2015} that we also have $C_r^*(G_U) = C^*(G_U)$.
Thus the second claim holds.

Now, let $U \subseteq \bN^*$ be a regular  open  set.
Then $U$ must take one of the following forms:
\begin{enumerate}
    \item $U\subseteq \bN$ and $U$ is finite;
    \item $\infty \in U$ and $\bN^*\backslash U$ is finite;
    \item $U \susbeteq \bN$ and both $U$ and $\bN\backslash U$ are infinite.
\end{enumerate}

In all cases, Claim 1 and Claim 2 show that we have the commutative diagram
\begin{equation}\label{comm diagram 3}
    \begin{CD}
    C^*(G_U) @>>> C^*(G) @>>> C^*(G_{\bN^* \backslash U})\\
    @| @| @| \\
    C_r^*(G_U) @>>> C_r^*(G) @>>> C_r^*(G_{\bN^* \backslash U}).
    \end{CD}
\end{equation}
As the top line of \eqref{comm diagram 3} is exact, the bottom line is also exact.
\end{proof}

\begin{example}\label{ex: HLS}
Let $\bF_2$ be the free group over two generators.
Let $(K_n)$ be the approximating sequence for $\bF_2$ given in \cite[Lemma~2.8]{Wil2015}, and let $G$ be the associated groupoid.
By \cite[Proposition~3.5]{Ana_weak_cont} $G$ is not exact, since $\bF_2$ is not amenable.
However, by \cite[Lemma~2.7 and Lemma~2.8]{Wil2015}, $C_r^*(G)= C^*(G)$.
Hence, by Proposition~\ref{prop: HLS reg} the failure of exactness for $G$ does not happen for open regular invariant sets $U \subseteq \go$.
\end{example}

\section{Applications}\label{applications}
In this section, we give a number of applications.
\subsection{Directed graphs}\label{ss: directed graphs 2}
Let $E$ be a directed graph.   In Remark~\ref{ss: directed graphs}, we showed that if $J$  is a gauge-invariant regular ideal in $C^*(E)$, then the graph $E/J$ satisfies condition~(L) if $E$ does.   In this subsection we use Corollary~\ref{cor: non-twisted} to show that if $E$ satisfies condition~(L) then all regular ideals in  $C^*(E)$ are gauge-invariant recovering  \cite[Proposition~3.7]{BFPR2021reg}.  

To see this, we use the groupoid $G$ constructed from $E$ in  \cite{KPRR1997} which, among other things, has  $C^*(E) \cong C_r^*(G)$  with the isomorphism sending each  vertex projection $p_v$  to a characteristic function on a compact open subset of $\go$. Further, the groupoid $G$ is amenable \cite[Corollary~5.5]{KPRR1997}, and so $G$ is exact, see e.g. \cite[Remark~4.5]{Ren1991}.

Since the gauge-invariant ideals of $C^*(E)$  are precisely the ideals generated by their vertex projections \cite[Theorem~4.1]{BPRS2000},
 the gauge-invariant ideals are ideals of the form $C_r^*(G_U)$ for open invariant sets $U \subseteq \go$.
As noted in Example~\ref{ss: directed graphs}, $D_E$ has the intersection property in $C^*(E)$ if and only if $E$ satisfies condition~(L). 
Thus Corollary~\ref{cor: non-twisted} generalizes \cite[Proposition~3.7]{BFPR2021reg}.
Similarly, for higher-rank graph \cstar-algebras, Corollary~\ref{cor: non-twisted} recovers \cite[Proposition~6.7]{Sch2021}.

In \cite{GonRoy2021} analogous results to the graph \cstar-algebra results of \cite{BFPR2021reg} are shown in the purely algebraic setting of Leavitt path algebras.
In a Leavitt path algebra the correct analogy of a gauge-invariant ideal is a $\bZ$-graded ideal.
It is shown in \cite[Corollary~3.4]{GonRoy2021} that all regular ideals in a Leavitt path algebra are graded, even when the graph does not satisfy condition~(L).
For graph \cstar-algebras, however, condition~(L) is needed.
Indeed, if $E$ is the graph of one vertex and one edge, then $C^*(E) = C(\bT)$.
Here any open regular subset of $\bT$ gives a regular ideal.
None of these, other than $\{0\}$ and $C(\bT)$, are gauge-invariant.

\subsection{Transformation groups}\label{ss: trans groups}
Let $\Gamma$ be a discrete group acting by homeomorphisms on a compact Hausdorff space $X$.
The action of $\Gamma$ on $X$ is \emph{topologically free} if the
interior of $X_t:= \{x\in X \colon t\cdot x = x\}$ is empty for all
non-identity elements $t \in \Gamma$.
The \iip\ for $(C(X)\rtimes_r \Gamma, C(X))$ is
closely related to topological freeness.  
In fact, when the action of $\Gamma$ on $X$ is topologically free,
S. Kawamura and J. Tomiyama show in~\cite[Theorem~4.1]{KawTom1990}
that $(C(X)\rtimes_r \Gamma, C(X))$ has the \iip, and the converse
holds when $\Gamma$ is amenable.  We do not know whether the \iip\ for
$(C(X)\rtimes_r \Gamma, C(X))$ is equivalent to topological freeness
of the action of $\Gamma$ on $X$ for discrete groups $\Gamma$ if the
hypothesis of amenability on $\Gamma$ is dropped.  However, some relaxation of
the amenability hypothesis is possible; see the equivalence of
statements (iv) and (v) of~\cite[Theorem~4.6]{PitZar2015}.

Recall that a closed set $F\subseteq X$ is \textit{regular} if it  is
the closure of its interior; it is easy to see that a closed set $F$ is
regular if and only if $X\setminus F$ is a regular open set.
We now observe that the restriction of a topologically free action to an invariant
regular closed set is again topologically free.     
\begin{proposition}\label{prop: reg and top free}
Let $\Gamma$ be a discrete group acting topologically freely
  on a compact Hausdorff space $X$.  If $Y\subseteq X$ is a regular
  and closed $\Gamma$-invariant set, then the restricted action of
  $\Gamma$ to $Y$ is also topologically free.
\end{proposition}

\begin{proof}
  Let $Y$ be an invariant regular closed set.  We need to show that
  the interior of $Y_s$ as a subset of $Y$, $\Interior{(Y_s)}{Y}$, is empty for all $e\neq s\in \Gamma$.  Suppose
  $s\in \Gamma$ and $\Interior{(Y_s)}{Y}\neq\emptyset$.  This means
  there exists an open set $U\subseteq X$ such that
  $U\cap Y\subseteq Y_s$ with $U\cap Y$ nonempty.  Since $Y$ is a
  regular closed set, $U\cap \interior{Y}\neq \emptyset$.  Thus,
\[\emptyset\neq U\cap \interior{Y}\subseteq U\cap Y\subseteq Y_s\subseteq X_s.\]  
This shows the interior of $X_s$ is non-empty.  Since $\Gamma$ acts
topologically freely on $X$, we conclude $s=e$.  \end{proof}

\begin{remark}
  Proposition~\ref{prop: reg and top free} can be used to construct an
  alternate proof of Theorem~\ref{lem: int prop quotient} for
  inclusions $C(X) \subseteq C(X) \rtimes_r \Gamma$.
           \end{remark}

Compare Proposition~\ref{prop: reg and top free} to Example~\ref{ex Disk}.
In Example~\ref{ex Disk} $X= \ol{\bD}$, and $F:=X\backslash \bP$ is a single point. The action of $\Gamma=\bZ$ on this point is (necessarily) trivial, and thus not topologically free.
In this case, however, $F$ is not a regular closed set.

\subsection{Cartan embeddings}
We used Theorem~\ref{lem: int prop quotient} to
show that the quotient of a Cartan inclusion by a regular ideal is
again a Cartan inclusion, see Theorem~\ref{thm: Cartan quotient}.  
We have already discussed in Section~\ref{ss: directed graphs 2} that Theorem~\ref{cor: U reg} recovers \cite[Proposition~6.7]{BFPR2021reg}.
Recall that condition~(L) in \cite[Proposition~6.7]{BFPR2021reg} is precisely the condition that the subalgebra $D_E$ is a Cartan subalgebra of the graph algebra $C^*(E)$.

Suppose $(A,D)$ is a Cartan inclusion and let $K\idealin D$ be an invariant regular ideal.
If $L \unlhd A$ satisfies $L \cap B =K = E(L)$, then $\<K\>_A \subseteq L \subseteq J_K$.
This leads to the obvious question:  does $\innerprod{K}_A=J_K$?
Equivalently, is $\<K\>_A$ a regular ideal of $A$?
In the graph algebra setting, \cite[Proposition~6.7]{BFPR2021reg} precisely says that we always have $\innerprod{K}_A=J_K$ for invariant regular ideals $K$.
We generalize this to more general Cartan inclusions now, using Theorem~\ref{cor: U reg}.

\begin{proposition}  \label{exactCartan}
Let $(A,D)$ be a Cartan inclusion with $A$ nuclear.
The map $J \mapsto J \cap D$ is a bijection between the regular ideals of $A$ and the invariant regular ideals of $D$.
The inverse map is given by $K \mapsto \<K\>_A$.
\end{proposition}

\begin{proof}
By \cite{Ren2008} (see also \cite[Corollary~7.6]{KwaMey2020}), there
  is a twist $\Sigma \rightarrow G$, with $G$ Hausdorff, \'etale, and
  effective so that $A$ is isomorphic to
  $C^*_r(\Sigma;G)$ via an isomorphism which carries $D$ to
  $C_0(\go)$.
By \cite[Theorem~5.4]{Tak2014} $G$ is amenable if and only if $A$ is
  nuclear.  Further, when $G$ is amenable, $G$ is exact, see,
  e.g. \cite[Remark~4.11]{Ren1991}. 

 Recall from Remark~\ref{rem: cartan iip} that $(A,D)$ has the \iip.
 Thus by applying 
Theorem~\ref{thm: 1-1 reg ideals} and Theorem~\ref{cor: U reg}, we obtain
the result. 
\end{proof}

\subsection{Crossed products of \cstar-algebras}
So far, the applications in this section have dealt with inclusions
$B \subseteq A$ where $B$ is abelian.  Here we sketch how our
arguments can be applied to reduced crossed products of possibly
nonabelian $C^*$-algebras by discrete exact groups to obtain an
analogue of Theorem~\ref{cor: U reg}.

Let $\Gamma$ be a discrete group acting by automorphisms on a \cstar-algebra $A$.
Let $A \rtimes_r \Gamma$ denote the reduced crossed product and let $E_A \colon A \rtimes_r \Gamma \rightarrow A$ be the usual faithful conditional expectation \cite[Proposition~4.1.9]{BrownOzawa}.
Recall from Example~\ref{ex: crossed prod}, that $E_A$ is invariant under $N = \{a \delta_s \colon s \in \Gamma\}$.
Note that, an ideal of $A$ is $N$-invariant if and only if it is invariant under the action of $\Gamma$.
By Corollary~\ref{cor: N-invariant is invariant} it then follows that an ideal in $A$ is invariant under the action of $\Gamma$ if and only if it is an invariant ideal in the sense of Definition~\ref{def invariant}.

If $K \unlhd A$ is a $\Gamma$-invariant ideal we get the following commutative diagram.

\begin{equation}
\label{comm diagram 2}
\xymatrix{
K \rtimes_r \Gamma \ar[r]\ar[d]^{E_K} & A \rtimes_r \Gamma\ar@{->>}[r]^\phi \ar[d]^{E_{A}} & A/K \rtimes_r \Gamma\ar[d]^{E_{A/K}}\\
 K \ar[r] & A\ar@{->>}[r]^\theta& A/K
}
\end{equation}
The bottom line of \eqref{comm diagram 2} will always be exact.
However, it is possible that the top line may not be exact.  As with
algebras associated with groupoids (see Definition~\ref{def: exact}),
if the top line of \eqref{comm diagram 2} is exact for all invariant
ideals $K \unlhd A$ we say that the action of $\Gamma$ on $A$ is
\emph{exact}.  In particular, this will happen if $\Gamma$ is exact.
See \cite{Sie2010} for a study of exact actions and the ideal
structure of $A \rtimes_r \Gamma$.

\begin{theorem}\label{cor: crossed prod ideals}
Let $\Gamma$ be a discrete group that acts by an exact action on a \cstar-algebra $A$.
If $K \unlhd A$ is an invariant regular ideal, then $K \rtimes_r \Gamma$ is a regular ideal in $A \rtimes_r \Gamma$.

If in addition $A \subseteq A \rtimes_r \Gamma$ has the ideal intersection property, every regular ideal in $A \rtimes_r \Gamma$ has this form.
\end{theorem}

\begin{proof}
  Let $K$ be an invariant regular ideal in $A$. Since the action of
  $\Gamma$ is exact we get that $\ker \phi=K\rtimes_r\Gamma$.  Arguing
  as in Proposition~\ref{prop: reg kern} one shows 
  $K\rtimes_r\Gamma=\ker \phi$ is a regular ideal, giving the first
  statement.

Additionally, assume that $A \subseteq A \rtimes_r \Gamma$ has the
ideal intersection property, and let $J_K$ be as in Notation~\ref{not:
  J_K}.  Arguing as in Proposition~\ref{Prop: J} we obtain $J_K=\ker \phi$.  Now Theorem~\ref{thm: 1-1 reg ideals} gives that all regular ideals are of this form. 
\end{proof}

Let $\Gamma$ be a discrete group acting by homeomorphisms on a compact space $X$.
The reduced crossed-product $C(X) \rtimes_r \Gamma$ is isomorphic to the reduced groupoid C$^*$-algebra $C_r^*(\Gamma \times X)$, where $\Gamma \times X$ is the transformation group.
Thus, crossed products of abelian C$^*$-algebras are special cases of groupoid C$^*$-algebras.
We have the following corollary  which is a special case of both Theorem~\ref{cor: crossed prod ideals} and Corollary~\ref{cor: non-twisted}.

\begin{corollary}\label{cor: trans groups}
Let $\Gamma$ be an exact discrete group acting on a compact Hausdorff space $X$ by homeomorphisms.
If $U \subseteq X$ is an open regular $\Gamma$-invariant subset, then $C_0(U) \rtimes_r \Gamma$ is a regular ideal in $C_0(X) \rtimes_r \Gamma$.

If the action of $\Gamma$ on $X$ is topologically free, then all regular ideals of $C_0(X)\rtimes_r \Gamma$ are of this form.
\end{corollary}

\begin{proof}
The groupoid $\Gamma \times X$ is exact since the group $\Gamma$ is exact.
This follows immediately on comparing Definition~\ref{def: exact} with Definition~1.5 of \cite{Sie2010}.
If the action of $\Gamma$ is topologically free, \cite[Theorem~4.1]{KawTom1990} shows $C(X)$ has the \iip\ in $C(X)\rtimes \Gamma$.
Thus, the result follows by Corollary~\ref{cor: non-twisted} or Theorem~\ref{cor: crossed prod ideals}.
\end{proof}

\begin{remark}
We compare Theorem~\ref{cor: crossed prod ideals} to results of Sierakowski \cite{Sie2010}.
The inclusion $A \subseteq A \rtimes_r \Gamma$ has the \emph{residual intersection property} if $A/J \subseteq A/J \rtimes_r \Gamma$ has the ideal intersection property for every invariant ideal $J \unlhd A$.
In \cite[Theorem~1.3]{Sie2010} it is shown that all ideals of $A \rtimes_r \Gamma$ are of the form $J \rtimes_r \Gamma$ for an invariant ideal $J \unlhd A$ if and only if the action of $\Gamma$ is exact on $A$ and $A \subseteq A \rtimes_r \Gamma$ has the residual intersection property.
Theorem~\ref{cor: crossed prod ideals} shows that the ideal intersection property suffices for regular ideals; the residual ideal intersection property is only needed for non-regular ideals.
\end{remark}

The question of when $A \subseteq A \rtimes_r \Gamma$ has the ideal intersection property has been studied by several authors.
 The ideal intersection property is determined by the action of $\Gamma$ on some injective objects.
However, instead of topological freeness, one must instead look to see if the action is \emph{properly outer}.
See \cite[Section~2.3]{Zar2019} or \cite[Definition~2.8]{KenSch2019} for definitions of properly outer actions.
Let $I(A)$ be the injective envelope of $A$ and let $I_\Gamma(A)$ be the $\Gamma$-injective envelope of $A$.
We summarize some of the results of Zarikian \cite{Zar2019} and Kennedy and Schafhauser \cite{KenSch2019}.

\begin{theorem}[{cf. \cite{Zar2019} and
    \cite{KenSch2019}}] \label{nonabel iip}
Let $\Gamma$ be a discrete group acting on a \cstar-algebra $A$.
Then the following are equivalent:
\begin{enumerate}
    \item $A \subseteq A \rtimes_r \Gamma$ has the \iip;
    \item $I(A) \subseteq I(A) \rtimes_r \Gamma$ has the \iip;
    \item $I_\Gamma(A) \subseteq I_\Gamma(A) \rtimes_r \Gamma$ has the \iip.
\end{enumerate}
In particular, this will happen if the action of $\Gamma$ on $I(A)$ is properly outer.
\end{theorem}

\begin{proof}
The equivalence of (i), (ii) and (iii) is \cite[Theorem~5.5]{KenSch2019}.
For the second part see \cite[Theorem~6.4]{Zar2019} or \cite[Theorem~6.4]{KenSch2019}.
\end{proof}

\section{Settings where the regular ideal intersection property
  and the \iip\ coincide}\label{RIIP and IIP}

For most of our theorems we require the seemingly weaker \riip\ instead of \iip. The main purpose of this section is to describe some classes of
inclusions for which the \iip\ and \riip\ coincide.
As noted in the introduction, the \iip\ has become an important
tool for obtaining structural results about inclusions, and it seems
to us that the equivalence of the \iip\ and the \riip\ is also an
interesting structural result in its own right.

We first give an example of an inclusion with the \riip\ but without
the \iip.

\begin{example}\label{ex: riip neq iip}
  Let $H$ be a separable Hilbert space.  Let $B \subseteq B(H)$ be any
  unital, non-zero \cstar-algebra having trivial intersection with the
  compact operators, $K(H)$.  The only ideals in $B(H)$ are $\{0\},\ K(H),$ and
  $B(H)$.  Of these, $\{0\}$ and $B(H)$ are regular ideals in $B(H)$, while $K(H)$ is
  not a regular ideal.     As $B(H) \cap B = B \neq \{0\}$, the inclusion
  $(B(H), B)$ has the \riip, but it does not have the \iip\ because 
  $K(H) \cap B = \{0\}$.

We can choose $B$ above so that $B$ is abelian and the inclusion is regular. For a straightforward
  example, the inclusion $(B(H), B)$ will be regular when $B = \bC I$
  since unitaries normalize $B$ and every $T\in B(H)$ is a linear
  combination of four unitary operators. 
\end{example}

We recall that a topological space is called \emph{semiregular} if it has a basis of regular open sets.

\begin{proposition} \label{prop: semireg RIIP}
Let $B \subseteq A$ be an inclusion of \cstar-algebras.
Assume that $\Prim(A)$  with the hull kernel topology is semiregular.
Then $B\subseteq A$ has the ideal intersection property if and only if it has the regular ideal intersection property.
\end{proposition}

\begin{proof}
Since the ideal intersection property implies the regular ideal intersection property, it suffices to show that for  $B\subset A$ the regular ideal intersection property implies the ideal intersection property.

Assume $B \subseteq A$ has the regular ideal intersection property and let $I \unlhd A$ be a non-trivial ideal.
Then $\Prim(A) \backslash \hull(I)$ is an open subset of $\Prim(A)$.
By the semiregularity of $\Prim(A)$ there is a regular open set $U \subseteq \Prim(A) \backslash \hull(I)$.
By Proposition~\ref{prop: bool}, the ideal  $J = \ker(\Prim(A) \backslash U)$ is a regular ideal $A$, and by \cite[Theorem~5.4.7~(3)]{MurphyBook} $J \subseteq I$.
By the regular ideal intersection property $J \cap B \neq \{0\}$, and hence $I \cap B \neq \{0 \}$.
Hence $B \subseteq A$ has the ideal intersection property.
\end{proof}

\begin{corollary}\label{cor: abelian RIIP}
Let $B \subseteq A$ be an inclusion of \cstar-algebras.
Assume that $\Prim(A)$  with the hull kernel topology is Hausdorff.
Then $B\subseteq A$ has the ideal intersection property if and only if it has the regular ideal intersection property.

In particular, if $A$ is abelian, then $B\subseteq A$ has the ideal intersection property if and only if it has the regular ideal intersection property.
\end{corollary}

\begin{proof}
The primitive ideal space $\Prim(A)$ is always locally compact.  Thus if $\Prim(A)$ is Hausdorff then $\Prim(A)$ is semiregular. This follows since locally compact Hausdorff spaces are regular, and regular spaces are semiregular, see \cite[pg. 16]{SteenSeebach}.
\end{proof}
\begin{remark}  By combining~Corollary~\ref{cor: abelian RIIP}
             with~\cite[Corollary~3.21]{PitZar2015} we obtain several
             checkable 
             characterizations of unital inclusions $B\subseteq A$
             with $A$ abelian 
             having the \riip\ or equivalently, the \iip.
\end{remark}

\begin{remark}
             We can apply Corollary~\ref{cor: abelian RIIP} to show
           that \riip\ and \iip\ are equivalent for inclusions in a
           variety of crossed products that are known to have
           Hausdorff spectrum.  For example, Williams characterizes
           when the crossed product of a transformation group
           $(\Gamma,X,\alpha)$ with $\Gamma$ second countable and
           abelian has Hausdorff spectrum \cite{Williams82}.  If the
           stability subgroups of $(\Gamma,X,\alpha)$ are all
           subgroups of a fixed abelian group, Williams characterizes
           when the crossed product is continuous trace (and so by
           definition has Hausdorff spectrum)
           \cite[Theorem~5.1]{Williams81}.  Echterhoff uses Williams'
           result to give conditions that guarantee crossed products
           are continuous trace for transformation groups
           $(\Gamma, X,\alpha)$ with $\Gamma$ a Lie group
           \cite[Corollary~3]{Echterhoff94} or $\Gamma$ is a discrete
           group \cite[Theorem~3]{Echterhoff94}.  Moreover, Archbold
           and an Huef \cite[Theorem~3.9]{Archbold_an_Huef08}
           characterize when the crossed product has continuous trace
           when the action of $\Gamma$ on $\Prim(A)$ is
           free.\footnote{We would like to thank the anonymous referee
             for pointing out the above applications.}  
           \end{remark}
           
           We state one
           application using a theorem of Green
           \cite{Gre1977}.

\begin{corollary}\label{cor: fplcgact}
Let $\Gamma$ be a locally compact group with with a free, proper
action on a locally compact Hausdorff space $X$.
Let $B \subseteq C_0(X) \rtimes \Gamma$ be a \cstar-subalgebra.
Then $B \subseteq C_0(X) \rtimes \Gamma$ has the ideal intersection property if and only if it has the regular ideal intersection property.
\end{corollary}

\begin{proof}
Note that the reduced crossed product $C(X) \rtimes_r \Gamma$ and the full crossed product $C(X) \rtimes \Gamma$ coincide since the action of $\Gamma$ is free and proper.
The quotient space $X/G$ is Hausdorff and $\Prim(C(X) \rtimes \Gamma) \cong X/G$ by \cite[Theorem~14]{Gre1977}.
The result thus follows from Corollary~\ref{cor: abelian RIIP}.
\end{proof}

The following application of Corollary~\ref{cor: abelian RIIP} gives a
wide variety of examples where the \riip\ is equivalent to the \iip.

\begin{theorem}\label{thm: iip v riip}
Let $(A,B)$ be a regular inclusion with $B$ abelian and assume that
$B^c$, the relative commutant of $B$ in $A$, is abelian.   If $(A,B^c)$
is a Cartan inclusion, then $(A,B)$ has the ideal intersection property if and only if $(A,B)$ has the regular ideal intersection property.
\end{theorem}

\begin{remark} Suppose $(A,B)$ is a regular inclusion with both $B$
  and $B^c$ abelian.  We wish to observe that $B^c\subseteq A$ is a
  regular inclusion of \cstaralg s.  An easy argument using the
  partial automorphism $\theta_n$ of
  \cite[Corollary~2.3]{Pit2021normalizers} shows
  $\N(A,B)\subseteq \N(A,B^c)$.  Therefore, $B^c$ is regular in $A$.
  Because $B$ contains an  approximate unit for $A$, so does $B^c$,
  whence $(A,B^c)$ is a regular
  inclusion.  Thus, in the setting of Theorem~\ref{thm: iip v riip},
  $(A,B^c)$ is a Cartan inclusion if and only if there is a faithful
  conditional expectation of $A$ onto $B^c$.
\end{remark}

\begin{proof}[\rm \textbf{Proof of Theorem~\ref{thm: iip v riip}}]
  To show that the \riip\ implies the \iip, we establish the
  contrapositive.  So assume  $B\subseteq A$ does not have the ideal
  intersection property and let  $\{0\}\neq J \unlhd A$ be
  such that $J \cap B = \{0\}$.   By hypothesis, 
  $B^c \subseteq A$ has the ideal intersection property.  Thus
  $J \cap B^c \neq \{0\}$.   Thus the inclusion
  $B \subseteq B^c$ does not have the ideal intersection property.
  By Corollary~\ref{cor: abelian RIIP}, $B \subseteq B^c$ does
  not have the regular ideal intersection property.

  Therefore, we may find a non-trivial regular ideal $K \unlhd B^c$ such
  that $K \cap B = \{0\}$.  By Theorem~\ref{thm: 1-1 reg ideals},
  there is a regular ideal
  $J_K \unlhd A$ such that
  $J_K\cap B^c = K$.  Hence
  $J_K \cap B = K \cap B = \{0\}$.  That is,
  $J_K\unlhd A$ is a non-trivial regular ideal with
  trivial intersection with $B$.  Hence $B \subseteq A$ does not have
  the regular ideal intersection property.
\end{proof}

We can readily apply Theorem~\ref{thm: iip v riip} to graph C$^*$-algebras.

\begin{corollary}\label{cor: graph iip}
Let $E$ be a row-finite graph.
Let $D_E$ be the abelain subalgebra of the graph C$^*$-algebra $C^*(E)$ described in Remark~\ref{ss: directed graphs}.
The following are equivalent:
\begin{enumerate}
\item $E$ satisfies condition~(L);
\item $D_E$ has the \iip\ in $C^*(E)$; and
\item $D_E$ has the \riip\ in $C^*(E)$.
\end{enumerate}
\end{corollary}

\begin{proof}
The commutant of $D_E$ in $C^*(E)$ is a Cartan subalgebra of $C^*(E)$ by \cite[{Theorem~3.6 and Theorem~3.7}]{NagRez2012}.
Thus the result follows from Theorem~\ref{thm: iip v riip} and the Cuntz-Krieger Uniqueness theorem \cite[Theorem~3.7]{KPR1998}.
\end{proof}

Our final corollary states that the \riip\ and  the \iip\ coincide for many groupoid \cstaralg s of interest, not just graph C$^*$-algebras.

\begin{corollary}\label{cor: riip groupoid}   
Suppose $G$ is an \'etale groupoid and let $G'$ be the isotropy
of $G$.   Assume that $G'$ is abelian and  the interior of
$G'$ is closed in $G$. 
Then $C_0(\go) \subseteq C_r^*(G)$ has the ideal intersection property if and only if it has the regular ideal intersection property.
\end{corollary}

\begin{proof}
Let $(G')^\circ$ be the interior of the isotropy of $G$.
  By \cite[Corollary~4.5]{BNRSW2016}, $C^*(\interior{(G')})$
  is Cartan in $C_r^*(G)$.
  Let $C_0(\unit G)^c$ be the relative commutant of $C_0(\go)$ inside $C_r^*(G)$.
 The result will follow   from
 Theorem~\ref{thm: iip v riip} once we show,
 \begin{equation}\label{rightcom} C_0(\unit
   G)^c=C^*_r(\interior{(G')}).
 \end{equation}
 Now Proposition II.4.7 (i) in \cite{RenaultBook} shows first that $C^*_r(\interior{(G')})\subseteq C_0(\unit
   G)^c$ and second that for any $b\in C_0(\unit
   G)^c$ we must have $b$  supported in $G'$.  But then the open support of $b$ must be contained in $\interior{(G')}$ and since $\interior{(G')}$ is closed we get $b\in C^*_r(\interior{(G')})$, that is  $C_0(\unit
   G)^c=C^*_r(\interior{(G')})$ as desired.\qedhere
 \end{proof}



\providecommand{\bysame}{\leavevmode\hbox to3em{\hrulefill}\thinspace}
\providecommand{\MR}{\relax\ifhmode\unskip\space\fi MR }
\providecommand{\MRhref}[2]{%
  \href{http://www.ams.org/mathscinet-getitem?mr=#1}{#2}
}
\providecommand{\href}[2]{#2}

\end{document}